\documentclass[11pt,letterpaper]{article}
\usepackage[margin=1in]{geometry}
\usepackage[T1]{fontenc}
\usepackage[utf8]{inputenc}
\usepackage[english]{babel}
\usepackage{microtype}
\usepackage[dvipsnames]{xcolor}
\usepackage{float}
\usepackage{graphicx}
\usepackage{caption}
\usepackage{subcaption}
\usepackage{soul}
\usepackage{enumitem}
\usepackage{manfnt}
\usepackage{ upgreek }
\usepackage[hang,flushmargin]{footmisc}

\definecolor{Crimson}{RGB}{153, 0, 0}
\definecolor{Scarlet}{RGB}{255, 36, 0}
\definecolor{Azure}{RGB}{0, 127, 255}
\definecolor{Navy}{RGB}{0, 0, 128}
\definecolor{Aranka}{RGB}{0, 135, 85}
\usepackage{xcolor}

\usepackage{amsmath, amsthm, amssymb, amsfonts, mathtools}
\usepackage{bbm}

\usepackage{tikz} 
\usetikzlibrary{arrows}
\usetikzlibrary{calc}
\usetikzlibrary{shapes}
\usetikzlibrary{positioning}
\usetikzlibrary{shapes.geometric}
\usetikzlibrary{arrows.meta}

\theoremstyle{plain}
\newtheorem{lemma}{Lemma}[section]
\newtheorem{theorem}[lemma]{Theorem}
\newtheorem{proposition}[lemma]{Proposition}

\newtheorem{corollary}[lemma]{Corollary}

\newtheorem{defi}[lemma]{Definition}

\theoremstyle{remark}
\newtheorem{remark}[lemma]{Remark}

\newtheorem{question}[lemma]{Question}

\newcommand{\Aut}{\mathrm{Aut}}

\newcommand{\Cay}{\mathrm{Cay}}
\newcommand{\Sch}{\mathrm{Sch}}
	
\newcommand{\acts}{\curvearrowright}

\newenvironment{acknowledgement}{\textbf{Acknowledgement.}}{}

\usepackage{hyperref}
\numberwithin{equation}{section}

\AtEndDocument{\bigskip{
  \noindent
  \textbf{Ferenc Bencs}\\
  Alfr\'ed R\'enyi Institute of Mathematics, Budapest, Hungary\\
  \href{mailto:bencs.ferenc@renyi.mta.hu}{\texttt{bencs.ferenc@renyi.mta.hu}} \par
  \addvspace{\medskipamount}
  
  \noindent
  \textbf{Aranka Hru\v{s}kov\'{a}}\\ 
  Central European University, Budapest, Hungary\\
  \href{mailto:aranka@renyi.mta.hu}{\texttt{aranka@renyi.mta.hu}} \par
  \addvspace{\medskipamount}
  
 \noindent
 \textbf{L\'{a}szl\'{o} M\'{a}rton T\'{o}th}\\
 École Polytechnique Fédérale de Lausanne, Switzerland\\
 \href{mailto:laszlomarton.toth@epfl.ch}{\texttt{laszlomarton.toth@epfl.ch}} 
}}
	
\begin{document}

\title{Factor-of-iid Schreier decorations of lattices in Euclidean spaces
\let\thefootnote\relax\footnotetext{\noindent\textit{Mathematics subject classification (2020):} 05E18, 05C63, 37A30, 60C05}}







\author{
  Ferenc Bencs\thanks{The first author is supported by the NKFIH (National Research, Development and Innovation Office, Hungary) grant KKP-133921.}
  \and
  Aranka Hrušková
  \and
  László Márton Tóth
}
\date{}

\maketitle
\begin{abstract}
    A Schreier decoration is a combinatorial coding of an action of the free group $F_d$ on the vertex set of a $2d$-regular graph. We investigate whether a Schreier decoration exists on various countably infinite transitive graphs as a factor of iid.
   
    We show that $\mathbb{Z}^d,d\geq3$, the square lattice and also the three other Archimedean lattices of even degree have finitary-factor-of-iid Schreier decorations,
    and exhibit examples of transitive graphs of arbitrary even degree in which obtaining such a decoration as a factor of iid is impossible.

    We also prove that symmetrical planar lattices with all degrees even have a factor of iid balanced orientation, meaning the indegree of every vertex is equal to its outdegree, and demonstrate that the property of having a factor-of-iid balanced orientation is not invariant under quasi-isometry.
\end{abstract}

\section{Introduction}

Let $G$ be a simple connected $2d$-regular graph. A \emph{Schreier decoration} of $G$ is a colouring of the edges with $d$ colours together with an orientation such that at every vertex, there is exactly one incoming and one outgoing edge of each colour.

It is a folklore result in combinatorics that every finite $2d$-regular graph has a Schreier decoration, and it is easy to extend this to infinite $2d$-regular graphs as well. As a generalisation of the finite result, the third author proved in \cite{toth2019invariant} that all $2d$-regular unimodular random rooted graphs admit an invariant random Schreier decoration. It is natural to ask whether such an invariant random Schreier decoration can be obtained as a factor of iid. In this article, we investigate this question for some infinite deterministic graphs, more specifically the $d$-dimensional Euclidean grids, symmetrical planar lattices, and graphs quasi-isometric to the bi-infinite path $P$ of the form $H\times P$. We study non-amenable graphs in a separate paper \cite{ourselves}.

Informally speaking, a Schreier decoration is a factor of iid if it is produced by a certain randomised local algorithm. To start with, each vertex of $G$ gets a random label from $[0,1]$ independently and uniformly. Then it makes a deterministic measurable decision about the Schreier decoration of its incident edges, based on the labelled graph that it sees from itself as a root. Neighbouring vertices must make a consistent decision regarding the edge between them. The factor is \emph{finitary} if the decision is based only on a finite-radius neighbourhood of each vertex. The precise definition is given in Section~\ref{section:basics}. Obtaining combinatorial structures or certain models in statistical mechanics as factors of iid is a central topic in ergodic theory. See~\cite{lyons2017factors} and the references therein for a recent overview. 

A partial result towards a Schreier decoration on a $2d$-regular graph is a balanced orientation of the edges. An orientation of the edges of a graph with all degrees even is \emph{balanced} if the indegree of any vertex is equal to its outdegree. For finite graphs, the term \emph{eulerian orientation} is often used \cite{schrijver1983bounds}, and when $G$ is 4-regular in particular, a balanced orientation is known as an ice configuration in statistical physics \cite{Welsh,Baxter}.
Every Schreier decoration gives a balanced orientation by forgetting the colours. 

The main results of the present paper are the following.

\begin{theorem}\label{thm:archimedean_schreier}
Let $\Lambda$ be any of the four Archimedean lattices with even degrees (pictured in Figure~\ref{fig:lattices}): the square lattice, the triangular lattice, the Kagomé lattice or the $(3,4,6,4)$ lattice.  
There is a finitary $\Aut(\Lambda)$-factor of $\left([0,1]^{V(\Lambda)},\mu_\Lambda\right)$ which is a.s.\ a Schreier decoration of $\Lambda$. Moreover, it has almost surely no infinite monochromatic paths.
\end{theorem}

\begin{theorem}\label{thm:Zd}
Let $\Lambda_{\square}^d$ denote the $d$-dimensional Euclidean grid.  For every $d\geq3$, there is a finitary $\Aut(\Lambda_{\square}^d)$-factor of iid which is a.s.\ a Schreier decoration of $\Lambda_{\square}^d$. Moreover, it has almost surely no infinite monochromatic paths.
\end{theorem}

As far as we are aware,
this had not been known, not even in the case of the square lattice.
Ray and Spinka show that balanced orientations (under the name ``6-vertex-model'') exist on the square lattice as finitary factors of iid~\cite[Remark 5.3.]{ray2019finitary}. Schreier decorations (that would be the ``24-vertex model'' in their terminology) are not investigated. In \cite{thornton2020factor}, Thornton studies factors of iid which are approximate and exact Cayley diagrams, but his results do not overlap with ours. 

Parts of the proof have to be adapted to the individual lattices, but our approach remains the same throughout.
We break the lattices into a hierarchy of finite pieces using percolation theory.
Then for each piece independently, we choose an edge-$d$-colouring scheme
such that we can ensure that every monochromatic connected subgraph is a finite cycle. Each cycle will orient itself strongly. We also use a similar hierarchy argument to find balanced orientations on symmetrical planar lattices; see Theorem~\ref{thm:planar_lattice_balanced_orientation}.

\begin{figure}[h]
    \centering
    \begin{subfigure}[p]{.3\linewidth}
        \includegraphics[width=\linewidth]{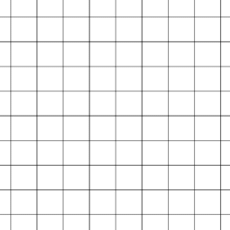}
        \caption{Square lattice}
    \end{subfigure}\qquad%
    \begin{subfigure}[p]{0.3\linewidth}
        \includegraphics[width=\linewidth]{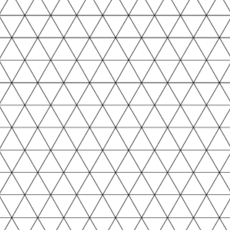}
        \caption{Triangular lattice}
    \end{subfigure}
    
    \vspace{2em}
    
    \begin{subfigure}[p]{.3\linewidth}
        \includegraphics[width=\linewidth]{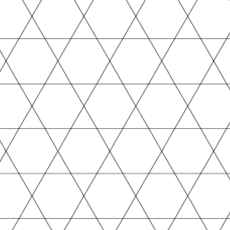}
        \caption{Kagom\'e lattice}
    \end{subfigure}\qquad%
    \begin{subfigure}[p]{0.3\linewidth}
        \includegraphics[width=\linewidth]{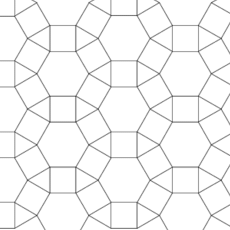}
        \caption{$(3,4,6,4)$ lattice}
    \end{subfigure}
    \caption{Archimedean lattices of even degree}
    \label{fig:lattices}
\end{figure}

Next we show examples in which it is impossible to obtain a Schreier decoration as a factor of iid.

\begin{theorem}\label{thm:counterexample}
For every $d \geq 1$, there exists a $2d$-regular transitive graph that has no factor of iid balanced orientation. In particular, it has no factor of iid Schreier decoration.
\end{theorem}

Theorem~\ref{thm:counterexample} answers \cite[Problem 3.4]{thornton2020orienting}. However, the examples we construct are somewhat unsatisfactory because they are all quasi-isometric to the bi-infinite path $P$. We do not know of any examples of transitive graphs with no factor of iid Schreier decoration that are not quasi-isometric to $P$. 

Furthermore, Theorems~\ref{thm:archimedean_schreier} and~\ref{thm:Zd} can be phrased in terms of Bernoulli graphings over the respective graphs, yielding the following corollary.

\begin{corollary}\label{cor:Bernoulli_grpahing}
The Bernoulli graphings of the Archimedean lattices as well as that of $\Lambda_\square^d$ admit a probability-measure-preserving action of the free group $F_d$ which satisfies that two vertices $x,y$ of the graphing are adjacent if and only if there is a generator $\gamma$ of $F_d$ such that $\gamma.x=y$.
\end{corollary}

This  connection is another important motivation for studying factor of iid Schreier decorations because we would like to understand which $2d$-regular graphings are so generated by actions of $F_d$. Our main question can be equivalently formulated as follows: given a (transitive unimodular) $2d$-regular graph $G$, is the Bernoulli graphing of $G$ generated by a p.m.p.\ action of $F_d$?
The connection is spelled out in our parallel paper \cite{ourselves}, where we investigate non-amenable graphs and utilise it in the other direction. Our main result there is the following.

\begin{theorem}[\cite{ourselves}]\label{thm:non-amenable_balanced_orientation}
Every non-amenable quasi-transitive unimodular $2d$-regular graph has a factor of iid that is almost surely a balanced orientation. 
\end{theorem}

For example, all Cayley graphs, in particular regular trees are unimodular.
For $d>1$, the $2d$-regular tree $T_{2d}$ is also non-amenable, so it is covered by Theorem~\ref{thm:non-amenable_balanced_orientation}.

The structure of the paper is as follows. In Section~\ref{section:basics}, we give definitions, present examples for Theorem~\ref{thm:counterexample}, and treat invariance under quasi-isometry.
Section~\ref{section:clusters_and_hierarchy} is concerned with general planar lattices, breaking them into clusters organised in a hierarchy and obtaining balanced orientations. Building on some of these results, Section~\ref{section:archimedean} gives the proofs of Theorem~\ref{thm:archimedean_schreier} separately for each lattice and of Theorem~\ref{thm:Zd}.
In Section~\ref{section:perfmatch_and_properedgecolr}, we explore what other combinatorial structures we can obtain using the results and ideas presented thus far.
Open questions are collected in Section~\ref{section:open_questions}.

\medskip

\begin{acknowledgement}
The authors would like to thank Miklós Abért, Jan Grebík, Matthieu Joseph, Gábor Kun, Gábor Pete, and Václav Rozhoň for inspiring discussions about various parts of this work.
\end{acknowledgement}

\section{Notation and basics}\label{section:basics}

We call a graph $G$ \emph{transitive} if and only if it is vertex-transitive, i.e.~the automorphism group $\Aut(G)$ acts transitively on the vertex set $V(G)$. We call $G$ \emph{edge-transitive} whenever $\Aut(G)$ acts transitively on $E(G)$.

For the rest of the paper, for every $d\geq2$, let $\Lambda_\square^d$ stand for the $2d$-regular $d$-dimensional grid, that is the Cayley graph of $\mathbb{Z}^d$ with the standard generators after forgetting the labelling. In the planar case, we also use simply $\Lambda_\square$ in place of $\Lambda_\square^2$.

\subsection{Schreier graphs}

Given a finitely generated group $\Gamma= \langle S \rangle$ and an action $\Gamma \acts X$ on some set $X$, the \emph{Schreier graph} $\Sch(\Gamma \acts X,S)$ of the action is defined as follows. The set of vertices is $X$, and for every $x\in X$, $s \in S$, we introduce an oriented $s$-labelled edge from $x$ to $s.x$.

Rooted connected Schreier graphs of $\Gamma$ are in one-to-one correspondence with pointed transitive actions of $\Gamma$, which in turn are in one-to-one correspondence with subgroups of $\Gamma$.
Trivially, a graph with a Schreier decoration is a Schreier graph of the free group $F_d$ on $d$ generators.
A special case is the
(left) \emph{Cayley graph} of $\Gamma$, denoted $\Cay(\Gamma,S)$, which is the Schreier graph of the (left) translation action $\Gamma \acts \Gamma$.

\subsection{Factors of iid}
Let $\Gamma$ be a group. A \emph{$\Gamma$-space} is a measurable space $X$ with an action $\Gamma \acts X$. A map $\Phi: X \to Y$ between two $\Gamma$-spaces is a \emph{$\Gamma$-factor} if it is measurable and $\Gamma$-equivariant, that is $\gamma.\Phi(x)=\Phi(\gamma.x)$ for every $\gamma \in \Gamma$ and $x \in X$.

A measure $\mu$ on a $\Gamma$-space $X$ is \emph{invariant} if $\mu(\gamma.A)=\mu(A)$ for all $\gamma\in\Gamma$ and all measurable $A \subseteq X$. We say an action $\Gamma \acts (X,\mu)$ is \emph{probability-measure-preserving} (p.m.p.) if $\mu$ is a $\Gamma$-invariant probability measure. 

Let $G$ be a graph and $\Gamma \leq \Aut(G)$. Let ${\tt u}$ denote the Lebesgue measure on $[0,1]$. We endow the space $[0,1]^{V(G)}$ with the product measure ${\tt u}^{V(G)}$. The translation action $ \Gamma \acts [0,1]^{V(G)}$ is defined by \[(\gamma.f)(v) = f(\gamma^{-1}.v), \ \forall \gamma \in \Gamma, v\in V(G).\]
The action $\Gamma \acts ([0,1]^{V(G)}, {\tt u}^{V(G)})$ is p.m.p.

An orientation of $G$ can be thought of as a function on $E(G)$ sending every edge to one of its endpoints. Viewed like this, orientations of $G$ form a measurable function space ${\tt Or}(G)$ on which $\Gamma$ acts. The set ${\tt BalOr}(G) \subseteq {\tt Or}(G)$ of balanced orientations is $\Gamma$-invariant and measurable, so it is a $\Gamma$-space in itself. Similarly, the set of all Schreier decorations of $G$ forms the $\Gamma$-space ${\tt Sch}(G)$. 

\begin{defi}
A $\Gamma$-factor of iid Schreier decoration (respectively, balanced orientation) of a graph $G$ is a $\Gamma$-factor $\Phi: ([0,1]^{V(G)}, {\tt u}^{V(G)}) \to {\tt Sch}(G)$ (respectively, to ${\tt BalOr}(G)$). If the group $\Gamma$ is not specified, we mean an $\Aut(G)$-factor.
\end{defi}

\begin{remark}
We allow $\Phi$ to not be defined on a ${\tt u}^{V(G)}$-null subset $X_0 \subseteq [0,1]^{V(G)}$. 
\end{remark}

Let us now recall some special classes of iid processes on graphs. For a fixed vertex $x \in V(G)$, let $\big(\Phi(\omega)\big)(x)$ denote the restriction of $\Phi(\omega)$ to the edges incident to $x$. We say $\Phi$ is a \emph{finitary} factor of iid if for almost all $\omega\in [0,1]^{V(G)}$, there exists an $R \in \mathbb{N}$ such that $\big(\Phi(\omega)\big)(x)$ is already determined by $\omega|_{B_G(x, R)}$. That is, if we change $\omega$ outside $B_G(x, R)$, the decoration $\Phi(\omega)$ does not change around $x$. This radius $R$ can depend on the particular $\omega$. If it does not then we say $\Phi$ is a \textit{block factor}.

When constructing factors of iid algorithmically, one often makes use of the fact that a uniform $[0,1]$ random variable can be decomposed into countably many independent uniform $[0, 1]$ random variables. In practice, this means that we can assume that a vertex has multiple labels or that a new independent random label is always available after a previous one was used.

We will use a reverse operation as well: the
composition of countably many uniform $[0,1]$ random variables is again a uniform $[0,1]$ random variable. When the number of variables combined is finite, one can do this in a permutation-invariant way. In particular, this allows finite graphs to make joint random decisions as factors of iid by composing the labels of their vertices. So any $\Aut(G)$-invariant random process on a finite graph $G$ is a factor of iid. We frequently make use of this on finite subgraphs of our infinite graphs. 

For $\Aut(G)$-factors of iid on an infinite transitive graph $G$, the fact that the factor map intertwines the actions of $\Gamma$ implies that the further two vertices are from each other, the more independently the process behaves around them. We make use of the existence of this correlation decay both later in this section and in Section~\ref{section:open_questions}.

\subsection{Graphs quasi-isometric to $P$}

In this subsection, we present regular graphs of arbitrary even degree with the same large-scale geometry which in one case have factor-of-iid Schreier decorations but in the other not even a factor-of-iid balanced orientation.
In both cases, our examples are quasi-isometric to the bi-infinite path $P$, that is the graph with $V(P)=\{v_i \colon i\in\mathbb{Z}\}$ in which $v_i$ and $v_j$ are adjacent if and only if $|i-j|=1$.

For simple graphs $G_1$ and $G_2$, let the graph $G_1 \times G_2$ be defined by having $V(G_1\times G_2)=V(G_1) \times V(G_2)$ with vertices $(u,v)$ and $(u',v')$ being adjacent if and only if $u=u'$ and $vv' \in E(G_2)$ or $v=v'$ and $uu' \in E(G_1)$.

\begin{proposition}\label{prop:finite_times_Z}
Let $H$ be a finite $2d-2$-regular graph with an odd number of vertices. The $2d$-regular graph $H \times P$ has no $\Aut(H \times P)$-factor of iid balanced orientation. 
\end{proposition}

\begin{proof}[Proof of Theorem~\ref{thm:counterexample}]
In Proposition~\ref{prop:finite_times_Z}, take $H$ to be the $2d-2$-regular clique $K_{2d-1}$. As $K_{2d-1}$ is transitive, so is $K_{2d-1} \times P$.
\end{proof}

\begin{proof}[Proof of Proposition~\ref{prop:finite_times_Z}]
First we note that $P$ has no $\Aut(P)$-factor of iid balanced orientation. This is because $P$ has only two balanced orientations, and so orienting any single edge determines the whole balanced orientation. So two edges at an arbitrary distance get oriented in the same direction with probability $1$. In a factor of iid orientation however, given two edges that are far enough from each other, the probability of them being oriented in the opposite direction is close to $1/2$.

Second, suppose a balanced orientation of $H \times P$ is given. For adjacent vertices $v,v'\in V(P)$, define \[n(v, v') = |\{u \in V(H) \mid \big((u,v),(u,v')\big) \textrm{ is an oriented edge in } H \times P\}|,\]
and note that $n(v_i,v_{i+1})+n(v_{i+1},v_i)=|V(H)|$.

We also claim that $n(v_i,v_{i+1})=n(v_{i+1},v_{i+2})$ for every $i\in\mathbb{Z}$. Indeed, as the orientation of $H \times P$ is balanced, we have 
\begin{equation}\label{eqn:indegree_outdegree}
\sum_{u \in V(H)} \mathrm{indeg}\big( (u,v_{i+1}) \big) = \sum_{u \in V(H)} \mathrm{outdeg}\big( (u,v_{i+1}) \big).
\end{equation}
On one hand, oriented edges of $H \times P$ of the form $\big((u,v_{i+1}),(u',v_{i+1})\big)$ contribute 1 to both sides of (\ref{eqn:indegree_outdegree}). On the other hand, $n(v_{i},v_{i+1})$ counts the edges of the form $\big((u,v_{i}),(u,v_{i+1})\big)$ contributing to the left-hand side of (\ref{eqn:indegree_outdegree}), and $n(v_{i+1},v_{i+2})$ counts the edges of the form 
$\big((u,v_{i+1}),(u,v_{i+2})\big)$ contributing to the right-hand side. Therefore $n(v_{i},v_{i+1})=n(v_{i+1},v_{i+2})$ as claimed.

As $|V(H)|$ is odd, we either have $n(v_{i},v_{i+1})>n(v_{i+1},v_{i})$ or $n(v_{i},v_{i+1})<n(v_{i+1},v_{i})$. We now have an argument similar to the one we presented to prove there is no factor of iid balanced orientation of $P$.

Our claim shows that if we have $n(v_{i},v_{i+1})>n(v_{i+1},v_{i})$ for some $i$, then we have it for all $i$. That is, for $i, j \in \mathbb{Z}$, and \emph{any} random balanced orientation, we have
\[\mathbb{P} \left[ n(v_{i},v_{i+1})>n(v_{i+1},v_{i}) \textrm{ and } n(v_{j},v_{j+1})<n(v_{j+1},v_{j}) \right] = 0.\]

However, if the balanced orientation of $H \times P$ was a factor of iid, for any $\varepsilon >0$ we could find $i$ and $j$ far enough such that
\[\left|\mathbb{P} \left[ n(v_{i},v_{i+1})>n(v_{i+1},v_{i}) \textrm{ and } n(v_{j},v_{j+1})<n(v_{j+1},v_{j}) \right] - \frac{1}{2} \right| \leq  \varepsilon.\]
Consequently, $H \times P$ has no factor of iid balanced orientation. 
\end{proof}

These non-examples are all quasi-isometric to $P$, which renders them somewhat unsatisfactory (see Question~\ref{qtn:non_Z_non_example}). Being quasi-isometric to $P$, however, does not imply being a non-example.

\begin{proposition}\label{prop:HxP_whenH_hasPM}
Let $H$ be a finite $2d-2$-regular graph which has a perfect matching. Then the $2d$-regular graph $H \times P$ has an $\Aut(H\times P)$-factor of iid Schreier decoration.
\end{proposition}
\begin{proof} 

Let $S$ be a non-empty factor-of-iid 4-independent subset of $P$, that is a subset such that for every $v_i, v_j\in S$, the distance $d_P(v_i,v_j)$ is at least 5.
For each $v\in S$, let $v'$ be a neighbour of $v$ in $P$, chosen uniformly at random. Then let both $H\times\{v\}$ and $H\times\{v'\}$ fix the same perfect matching $M$. Together with the $|V(H)|$ edges between $H\times\{v\}$ and $H\times\{v'\}$, these form a 2-factor of $H\times\{v,v'\}$ (whose all cycles are even).

Next we are going to complement these 2-factors to obtain a 2-factor of the entire $H\times P$. Let $S'=\{v' : v\in S\}$, and suppose that $i\in\mathbb{Z}, j\in\mathbb{N}$ are such that $v_{i-1},v_{i+j+1}\in S\cup S'$ but $v_k\notin S\cup S'$ for any $k\in[i,i+j]$. We note that $j$ is finite and greater than zero because $S$ is a factor of iid and 4-independent, respectively.
Now let $H\times\{v_i\}$ fix the same perfect matching as $H\times\{v_{i-1}\}$ does and $H\times\{v_{i+j}\}$ the same as $H\times\{v_{i+j+1}\}$.
Together with the $j\cdot|V(H)|$ edges of the form $(u,v_{i+k-1})(u,v_{i+k}), k\in[j]$, these form a 2-factor of $H\times\{v_i,\dots v_{i+j}\}$ (whose all cycles are even).

Let us give colour $c_1$ to every edge in this 2-factor of $H\times P$ and orient each of its finite cycles strongly. Then after removing all the decorated edges, we are left with a $2d-2$-regular graph whose all connected components are finite; in particular, they are all isomorphic to either $H$ or a connected component of $(H\setminus M) \times \{v_i,v_{i+1}\}$.
To complete the construction, let each connected component pick at random a Schreier decoration with colours $c_2,\dots,c_{d}$.
\end{proof}

With a couple more technicalities in the proof, we could in fact tweak the Schreierisation so that all oriented cycles of colour $c_1$ would have length at most $3\cdot|V(H)|$ and all oriented cycles of the other colours at most $2\cdot|V(H)|$. This means that $H\times P$ is the Schreier graph of a factor-of-iid action of many more groups than just $F_d$.

Most $2d-2$-regular graphs with an even number of vertices do have a perfect matching. For example, if they are bipartite or contain a Hamiltonian path, a perfect matching must exist. However, there are instances for every $d\geq3$ which do not have a perfect matching despite having an even number of vertices \cite{perfectmatch_in4,perfectmatch_inregular}. For these, we do not know whether their product with $P$ admits a factor-of-iid Schreier decoration and/or a perfect matching; see also subsection~\ref{subsec:properedgeperfmatch} and Question~\ref{qtn:characterisation}.

Finally, taking $H=K_{2d-1}$ in Proposition~\ref{prop:finite_times_Z} and $H=K_{2d-2,2d-2}$ in Proposition~\ref{prop:HxP_whenH_hasPM} shows that the property of having a factor-of-iid Schreier decoration is not invariant under quasi-isometry not even when we restrict to transitive graphs of a given even degree.

\section{Finite clusters of arbitrary thickness and their hierarchy}
\label{section:clusters_and_hierarchy}

In this section, we will develop the main tools to investigate planar lattices. Our main goal is to break the symmetries of the lattice by breaking the lattice into pieces of finite size such that each of the pieces is ``wide'' enough and surrounded by another one -- such a partition will be called a hierarchy. We use percolation theory to obtain a starting hierarchy as a factor of iid process. We also prove that given any starting hierarchy, one can coarsen it in a factor of iid way so that the pieces of the resulting hierarchy have arbitrarily large width.
As an application, we show how to obtain a factor of iid balanced orientation of any planar lattice with $m$-fold symmetry and all degrees even.

Throughout this and the following section, for any graph $G$, let $\mu_G$ denote the
usual product measure on $[0,1]^{V(G)}$, with each coordinate getting the Lebesgue measure.

\subsection{General cluster hierarchy and its thickness}

To obtain Schreier decorations of three of the Archimedean lattices, we partition their vertex set $V$ into finite parts, which we shall call clusters, such that for each cluster $C$, there is a unique cluster $C^+$ surrounding it. That is, a unique part $C^+$ such that $C$ is in its convex hull and such that there are adjacent vertices $u\in C$, $v\in C^+$. Such a partition of $V$ can be described by a one-ended infinite tree $T$ whose vertices are the finite clusters; $B, C\in V(T)$ are adjacent if and only if one surrounds the other. Moreover, every vertex $C\in V(T)$ has a well-defined parent $C^+$ distinguishable from $C$'s children, should there be any.
\begin{defi}[Hierarchy]
Let $G$ be a graph and $\mathbf{H}$ a partition of $V(G)$. We say that two distinct parts $C, D \in \mathbf{H}$ are adjacent if and only if there is $u\in C$ and $v\in D$ which are adjacent in $G$. Then $\mathbf{H}$ is a \emph{hierarchy on $G$} if the following holds for every $C\in \mathbf{H}$.\\
1) $C$ is finite,\\
2) there is a \emph{unique} $C^+\in \mathbf{H}$ such that $C$ and $C^+$ are adjacent and for all $v\in V(G)$ but finitely many, any path from $C$ to $v$ contains a vertex from $C^+$,\\
3) whenever $B\in \mathbf{H}$ is adjacent to $C$, either $B=C^+$ or $C=B^+$.
\end{defi}

Note that it is not necessarily the case that the subgraph of $G$ induced by a cluster $C$ is connected.

A key feature of the hierarchies we use later is that any two non-adjacent clusters are far from one another. Let us describe how starting with any hierarchy, we can obtain one in which non-adjacent clusters are as far from one another as we wish.

\begin{defi}[$k$-spaced hierarchy]
Let $G$ be a graph and $k$ a natural number. A hierarchy $\mathbf{H}$ on $G$ is \emph{$k$-spaced} if for all non-adjacent $B,C\in \mathbf{H}$, the graph distance $d(B,C)=\min_{u\in B, v\in C} d_G(u,v)$ is at least $k$.
\end{defi}

\begin{proposition}\label{prop:hierarchy}
Let $G$ be a graph and suppose there is a (finitary) $\Aut(G)$-factor of iid hierarchy $\mathbf{H}$ on $G$. Then $\forall k\in\mathbb{N}$, there is a (finitary) $\Aut(G)$-factor of iid $k$-spaced hierarchy $\mathbf{H}_k$ on $G$.

Moreover, $\forall c,k\in\mathbb{N}$, there is a (finitary) $\Aut(G)$-factor of iid pair $(J_{c,k}$,$\eta: J_{c,k}\to[c])$ where $J_{c,k}$ is a $k$-spaced hierarchy and $\eta: J_{c,k}\to[c]$ is a colouring of the parts with $c$ colours such that $\forall C\in J_{c,k}$, if $C$ has colour $i$ then $C^+$ has colour $i+1\pmod{c}$. 
\end{proposition}
\begin{proof}
Consider the infinite tree $T_\mathbf{H}$ whose vertices are the parts of $\mathbf{H}$, i.e.~the clusters of the hierarchy.
Let us now colour uniformly randomly the vertices of $T_\mathbf{H}$ green and yellow, independently of each other (in other words, we carry out a site percolation with $p=\frac{1}{2}$). For any vertex $z\in V(T_\mathbf{H})$, write $z^+$ for the parent of $z$ and $z^{n+}\vcentcolon=(z^{(n-1)+})^+$ whenever $n\geq2$. Then with probability 1, $\forall z\in V(T_\mathbf{H})$ $\exists n$ such that $z^{n+}$ has the opposite colour than $z$.
Together with the fact that every vertex $x\in V(T_\mathbf{H})$ is an ancestor of only finitely many other vertices (i.e.~$| \{y\in V(T_\mathbf{H}) : \exists n \text{ such that } x=y^{n+} \} |<\infty$), this means that with probability 1, all monochromatic connected components of the coloured $T_\mathbf{H}$ are finite.

For every $z\in V(T_\mathbf{H})$, let us now merge it with $z^+$ if and only if $z$ and $z^+$ have got the same colour or $z$ is yellow and $z^+$ is green. The modification this makes on the tree $T_\mathbf{H}$ is that of contracting all edges whose two endpoints have the same colour and those edges $zz^+$ for which $z$ is yellow and $z^+$ is green.
This corresponds to
coarsening $\mathbf{H}$ into a new
hierarchy $\mathbf{H}^\wr$ which has got the property that for each cluster $C$, the distance $d(C,C^{++})=\min_{u\in C, v\in C^{++}}d_G(u,v)$ is at least 3.
This is because every path from $C$ to $C^{++}$ must go through $C^+$, but $C^+\in \mathbf{H}^\wr$ consists of at least two clusters of the original hierarchy $\mathbf{H}$.
More precisely, $C$ is the finite union of some former clusters $A_1,\dots,A_m\in \mathbf{H}$, $C^+$ of $B_1,\dots,B_n$ and $C^{++}$ of $D_1,\dots,D_l$. $C^+$ surrounds $C$, so there must be unique $i\in[m], j\in[n]$ such that $A_i^+=B_j$. $C$ and $C^+$ being distinct clusters in $\mathbf{H}^\wr$ implies that the vertex $A_i\in V(T_\mathbf{H})$ was coloured green and $B_j\in V(T_\mathbf{H})$ yellow. But then $B_j$ cannot be the part of the union $\cup_r^nB_r$ through which $C^+$ neighbours $C^{++}$, that is $B_j^+\in\{B_1,\dots,B_n\}$, and so $d(C,C^{++})=d(A_i,C^{++})\geq d(A_i,B_j^{++})=d(A_i,A_i^{+++})\geq 3$.

If we repeat such contraction of the hierarchy tree $m$ times to obtain $\mathbf{H}^{m\wr}$, then $d(C,C^{++})\geq2^m+1$ for every $C\in \mathbf{H}^{m\wr}$.
However, for $B, C\in \mathbf{H}^{m\wr}$ such that $B^+=C^+$, we can still have $d(B,C)=2$. To fix this, let us divide every $C\in \mathbf{H}^{m\wr}$ to

\begin{align*}
C_{\rm outer}&=\{v\in C : d(v,C^+)=\min_{u\in C^+} d_G(v,u)<2^{m-1}\}\\
\text{and }
C_{\rm inner}&=\{v\in C : d(v,C^+)=\min_{u\in C^+} d_G(v,u)\geq2^{m-1}\},
\end{align*}
and define a hierarchy $\mathbf{H}_{2^{m-1}}=\{B_{\rm inner}\cup \bigcup_{B=C^+}C_{\rm outer} : B\in \mathbf{H}^{m\wr}\}$.
Then for every $B\in \mathbf{H}^{m\wr}$ such that $B_{\rm inner}\neq\emptyset$,
\begin{align*}
   d&\Big(B_{\rm inner}\cup \bigcup_{B=C^+}C_{\rm outer},\big(B_{\rm inner}\cup \bigcup_{B=C^+}C_{\rm outer}\big)^{++}\Big)\\
    =d&\Big(B_{\rm inner}\cup \bigcup_{B=C^+}C_{\rm outer},B^{++}_{\rm inner}\cup \bigcup_{B^{++}=C^+}C_{\rm outer}\Big)\\
    =d&\Big(B_{\rm inner},\bigcup_{B^{++}=C^+}C_{\rm outer}\Big)
    =d(B_{\rm inner},B^+_{\rm outer})
    \geq d(B_{\rm inner}, B^+)=2^{m-1}.
\end{align*}
Two clusters in $\mathbf{H}_{2^{m-1}}$ have got the same parent if and only if they are of the form $B_{\rm inner}\cup \bigcup_{B=C^+}C_{\rm outer}$ and $D_{\rm inner}\cup \bigcup_{D=C^+}C_{\rm outer}$ for some $B,D\in \mathbf{H}^{m\wr}$ with $B^+=D^+$. For such a pair, we observe that
\begin{align*}
    d\Big(B_{\rm inner}\cup \bigcup_{B=C^+}C_{\rm outer},D_{\rm inner}\cup \bigcup_{D=C^+}C_{\rm outer}\Big)
    &=d(B_{\rm inner},D_{\rm inner})\\
    & \geq d(B_{\rm inner}, B^+)+d(D^+,D_{\rm inner})=2\cdot2^{m-1}.
\end{align*}

Having obtained $\mathbf{H}_k$ with $d(B,C)\geq k$ whenever $B, C$ are non-adjacent for every $k\in\mathbb{N}$, we want to find coloured hierarchies $J_{c,k}$. For fixed $c, k\in\mathbb{N}$, start by considering a hierarchy $\mathbf{H}_{ck}$. Then divide every $C\in \mathbf{H}_{ck}$ to parts $C_1,\dots,C_c$ where
\begin{align*}
    C_i&=\{v\in C : d(v,C^+)\in((i-1)k,ik]\} \text{ for every } i\in[c-1]\\
    \text{and } C_c&=\{v\in C : d(v,C^+)>(c-1)k\}.
\end{align*}
Colouring $C_i$ with colour $c+1-i$ yields the desired pair $(J_{c,k},\eta: J_{c,k}\to[c])$.
\end{proof}

\begin{figure}
    \centering\resizebox{0.8\linewidth}{!}{
    \begin{tikzpicture}   
    \node (anchor) at (0,0) {\includegraphics[width=12cm]{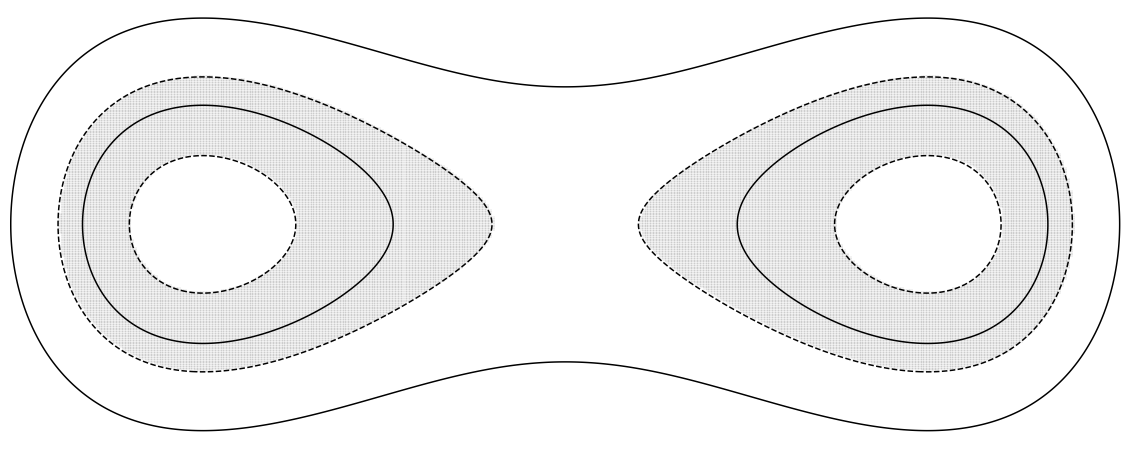}};
    \node (Ci) at (0,-1) {$C_\textrm{outer}$};
    \node (Co) at (0,1) {$C_\textrm{inner}$};
    \draw[->] (Co) edge (-1.5,0.5);
    \draw[->] (Co) edge (1.5,0.5);
    \node (Bi) at (-3.75,0) {$B_\textrm{inner}$};
    \node (Bi2) at (3.75,0) {$B'_\textrm{inner}$};
    \node (Bo) at (-5,-2.5) {$B_\textrm{outer}$};
    \node (Bo2) at (5,-2.5) {$B'_\textrm{outer}$};
    \draw[->] (Bo) edge (-4,-1);
    \draw[->] (Bo2) edge (4,-1);
    \end{tikzpicture}}
    \caption{Splitting clusters $C$, $B$, and $B'$ into their inner and outer parts. The shaded region shows $C_\textrm{inner}\cup \bigcup_{B^+=C}B_\textrm{outer}$.}
    \label{fig:cluster_eating}
\end{figure}
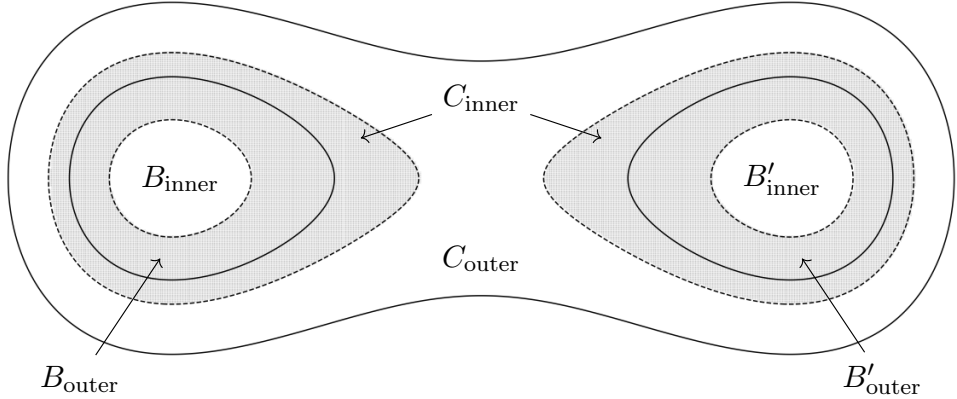
\begin{remark}
The Proposition above can be generalised to cases in which siblings are allowed to communicate, i.e.~in which two clusters $B\neq C$ with $B^+=C^+$ might be neighbours.
This corresponds to modifying the third point in the definition of hierarchy or leaving it out altogether. Under these circumstances, the infinite graph describing the hierarchy is no longer a tree, but a tree-like (or even forest-like) structure with triangles allowed.
\end{remark}

\subsection{Percolation clusters in planar lattices and their matching lattices}\label{subsec:percolationplanar}

Let $\Lambda$ be a \emph{planar lattice}, that is a connected, locally finite plane graph, with $V(\Lambda)$ a discrete subset of $\mathbb{R}^2$, such that there are translations $T_{v_1}$ and $T_{v_2}$ through two independent vectors $v_1$ and $v_2$ both of which act on $\Lambda$ as a graph isomorphism \cite{BollRior}.
Note that any planar lattice $\Lambda$ is necessarily quasi-transitive. We wish to use site percolation in a way that would partition $V(\Lambda)$ into finite clusters with a hierarchy. For lattices satisfying $\theta^s\left(\frac{1}{2}\right)=0$, i.e.~on which site percolation does not occur when $p=\frac{1}{2}$, we could colour the vertices $V(\Lambda)$ uniformly at random yellow and green and consider the monochromatic connected components. However, though this produces, with probability 1, only finite clusters, there isn't necessarily a clear hierarchy associated to them; this inconvenience is best illustrated by the random vertex 2-colouring of the square lattice. Nevertheless, when the lattice possesses some mild symmetry properties, a two-step solution presents itself. The first step is a result on matching pairs with $m$-fold symmetry.
\begin{defi}[$m$-fold symmetry, \cite{matchingpairs}]
For $m\geq2$, a plane lattice $\Lambda$ has \emph{$m$-fold symmetry} if the rotation about the origin through an angle of $2\pi/m$ maps the plane graph $\Lambda$ into itself.
\end{defi}
Based on ideas of Zhang, Bollobás and Riordan prove the following \cite{BollRior,matchingpairs}.
\begin{theorem}\label{thm:matchlattice}
Let $\Lambda$ be a plane lattice with $m$-fold symmetry for some $m\geq2$ and $\Lambda^\times$ its matching lattice, i.e.~the graph obtained from $\Lambda$ by adding all diagonals to all faces of $\Lambda$. Then for every $p\in[0,1]$, the percolation probabilities $\theta_\Lambda^s(p)$, $\theta_{\Lambda^\times}^s(p)$ satisfy that $\theta_\Lambda^s(p)=0$ or $\theta_{\Lambda^\times}^s(p)=0$. Furthermore, $p_H^s(\Lambda)+p_H^s(\Lambda^\times)=1$, where $p_H^s$ is the Hammersley critical probability.
\end{theorem}
This result tells us that choosing to randomly 2-colour the vertices of our planar lattice $\Lambda$ such that a site is yellow with probability $p_H^s(\Lambda)$ and green with probability $p_H^s(\Lambda^\times)$ is well defined. The definition of a cluster hierarchy on $\Lambda$ then comes naturally, with yellow clusters being exactly the yellow connected components of the random 2-colouring and green clusters the would-be green connected components on $\Lambda^\times$ (that is in $\Lambda$, the green clusters are unions of connected components which have vertices appearing in a same face).
The trouble is that the theorem does not guarantee that \emph{neither} of the colours will have an infinite cluster, or in other words that percolation does not occur at criticality at \emph{neither} the planar lattice nor its matching lattice. Benjamini and Schramm conjectured in 1996 that this indeed \emph{is} the case \cite{beyond}. However, for now, to obtain this crucial property, one needs to analyse lattices one by one (e.g. Russo showed that both critical percolations die for the square lattice $\Lambda_\square$ and its matching lattice $\Lambda_\boxtimes$ \cite{Russo}).

The second step in our solution to the hierarchy problem is therefore to add a vertex to every non-triangular face of a lattice $\Lambda$ with $m$-fold symmetry and connect it to all the vertices of that face. Let us call the resulting lattice $\Lambda^\bullet$ and observe that it also has $m$-fold symmetry. $\Lambda^\bullet$ is self-matching, and so Theorem~\ref{thm:matchlattice} tells us that $p^s_H(\Lambda^\bullet)=\frac{1}{2}$ and percolation does not occur at criticality.

This finding gives us a factor-of-iid hierarchy on $\Lambda$ as follows. Colour the vertices of $\Lambda$ yellow or green uniformly at random. For each non-triangular face, decide uniformly at random whether either all of its yellow vertices will be treated as if they were connected through the face or all of its green vertices will be so treated. This results, with probability one, in a well-defined cluster hierarchy on $\Lambda$, which we can combine with Proposition~\ref{prop:hierarchy} to conclude the following.

\begin{theorem}\label{thm:latticehierarchy}
Let $\Lambda$ be a planar lattice with $m$-fold symmetry, $m\geq2$. Then for every $k\in\mathbb{N}$, there is a finitary $\Aut(\Lambda)$-factor of $\left([0,1]^{V(\Lambda)},\mu_\Lambda\right)$ which is almost surely a $k$-spaced hierarchy on $\Lambda$.
\end{theorem}

Theorem~\ref{thm:latticehierarchy} now allows us to obtain a finitary factor of iid balanced orientation of any planar lattice with $m$-fold symmetry in which all degrees are even. To carry out the construction, we first need to clearly demarcate the clusters.

\begin{defi}[Boundary $\partial B$]
Let $\Lambda$ be a planar lattice and $\mathbf{H}$ a hierarchy on $\Lambda$. The \emph{boundary $\partial B$} of a cluster $B\in \mathbf{H}$ is a set of edges with both endpoints in $B$ as follows.
An edge $uv\in E(B)$ is in $\partial B$ if and only if one of its two contiguous faces $F_1, F_2$ consists only of vertices of the cluster $B$ and its offspring, while the other contains a vertex from $B^+$.
\end{defi}

Note that in general, it may be the case that $F_1=F_2$, but when all degrees of $\Lambda$ are even, these two faces must be distinct.

\begin{theorem} \label{thm:planar_lattice_balanced_orientation}
Let $\Lambda$ be a planar lattice with $m$-fold symmetry, $m\geq2$, in which all degrees are even. There is a finitary $\Aut(\Lambda)$-factor of $\left([0,1]^{V(\Lambda)},\mu_\Lambda\right)$ which is a balanced orientation of $\Lambda$ almost surely.
\end{theorem}

\begin{proof}
$V(\Lambda)$ is a discrete set in $\mathbb{R}^2$, i.e.~it has no accumulation points, and so there is $\ell\in\mathbb{N}$ such that every face of $\Lambda$ consists of at most $\ell$ vertices. Let us fix a finitary factor of iid $\frac{\ell+1}{2}$-spaced hierarchy $\mathbf{H}$ on $\Lambda$, so that no face contains vertices of non-adjacent clusters, and consider the boundaries $\partial B$, $B\in \mathbf{H}$.

We observe that for every vertex $v\in B$, the number of edges from $\partial B$ that are incident to $v$ is even. Indeed, let $r$ be a positive number such that the disc $D$ with radius $r$ centred on $v$ contains no other vertices than $v$, and let $x_1,\dots,x_{\deg (v)}$ be the neighbours of $v$ as we traverse them clockwise. Then the edges $vx_1,\dots,vx_{\deg (v)}$ split $D$ into $\deg (v)$ many areas $A_{i-1,i}$, where $vx_{i-1}$ moves through $A_{i-1,i}$ clockwise towards $vx_i$ (all indices are taken$\pmod{\deg (v)}$). Now call $F_{i-1,i}$ the faces spanned by $A_{i-1,i}$; note that we might have $F_{i-1,i}=F_{j-1,j}$ even when $i\neq j$. The faces $F_{i-1,i}$ are of two types -- those with all vertices from the cluster $B$ and its offspring and those with a vertex from $B^+$.
As no face contains vertices of non-adjacent clusters, the edge $vx_i$ is in $\partial B$ if and only if the two consecutive faces $F_{i-1,i}, F_{i,i+1}$ are of different types. But when we start in $F_{1,2}$ and make one full circle clockwise back to $F_{1,2}$, the change of types must occur even number of times.

Even though the construction of a spaced hierarchy $\mathbf{H}$ on $\Lambda$ based on Proposition~\ref{prop:hierarchy} and Theorem~\ref{thm:latticehierarchy} does not ensure that all clusters are connected, it does say that $\cup_{n=1}^\infty C^{n+}$ is connected (and infinite) for every $C\in \mathbf{H}$. Therefore, there is no vertex from $\cup_{n=1}^\infty C^{n+}$ inside $\partial C$ (each cycle of $\partial C$ splits the plane into a finite and an infinite region; by inside, we mean the union of the finite regions).
On the other hand, all vertices of any offspring of $C$ are inside $\partial C$, so any cluster which has a child must have non-empty boundary. Let us now reassign any vertices of a cluster $C$ which are outside $\partial C$ to $C^+$ -- these are exactly the vertices such that all faces they are in contain a vertex from $C^+$.

Every cluster $C$ will now randomly choose one of two allowed orientation patterns which will be given to all non-boundary edges with at least one endpoint in $C$ and none in $C^+$. These are exactly the edges in the finite region enclosed by $\partial C$, but outside the finite regions enclosed by $\partial B$ for all $B$ with $B^+=C$.

The patterns are as follows: for a vertex $v$ and its neighbours $x_1,\dots,x_{\deg(v)}$, ordered as we traverse them clockwise, we want the orientations of $vx_1,\dots,vx_{\deg(v)}$ to be alternating between into $v$ and out of $v$. In other words, we choose randomly one of the two chessboard colourings of the faces in this region, and orient cycles bounding black faces clockwise and cycles bounding white faces counter-clockwise.

For every cluster $C$, if $C$ and $C^+$ happen to choose the same pattern, we will propagate it to $\partial C$ as well. If $C$ and $C^+$ choose different patterns, then $\partial C$ randomly picks one of its balanced orientations (recall that $\partial C$ has all degrees even, so a balanced orientation exists).

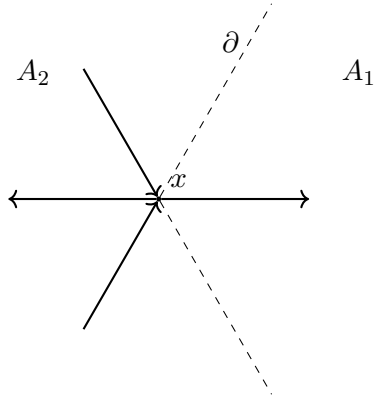
\begin{figure}[h]
    \centering
    \begin{tikzpicture}[node distance=1cm,
        main node/.style={circle,minimum width=1.9pt, fill, inner sep=0pt,outer sep = 0pt},
        red node/.style={circle,minimum width=2.2pt, fill,red, inner sep=0pt,outer sep = 0pt},
        blue node/.style={circle,minimum width=2.2pt, fill,blue, inner sep=0pt,outer sep = 0pt},
        square node/.style={draw,regular polygon,regular polygon sides=4,fill,inner sep=0.9pt,outer sep=0.1pt},
        triangle node/.style={draw,regular polygon,regular polygon sides=3,fill,inner sep=0.8pt,outer sep=0.1pt}]  
        
        \node[main node] (0) at (0,0) [label=above right:{$x$}] {};
        
        \draw[dashed] (0) -- (60:3) node[pos=0.8, left] {$\partial$};
        \draw[dashed] (0) -- (300:3);
        
        \node (a1) at (30:2.5) [label=above right:{$A_1$}] {};
        \node (a2) at (150:2.5) [label=above right:{$A_2$}] {};

        \draw[->, thick] (0) -- (0:2);
        \draw[<-, thick] (0) -- (120:2);
        \draw[->, thick] (0) -- (180:2);
        \draw[<-,thick] (0) -- (240:2);
    \end{tikzpicture}

    \caption{A disagreement between patterns around a vertex of degree 6}
    \label{fig:lattice_balanced_orientation}
\end{figure}

Finally, we claim that the resulting global orientation of $\Lambda$ is balanced. If a vertex $v\in C$ is not in $\partial C$ or if $C$ and $C^+$ agree on their patterns, then all the edges incident to $v$ follow the same pattern, i.e.~they alternate between in and out and the orientation is balanced at $v$. Now suppose $C$ and $C^+$ have got different patterns and some of the edges incident to $v$ are in $\partial C$. Then let $D$ be a disc centred on $v$ as before -- the boundary edges split it into areas $A_0,\dots,A_{2n-1}$, ordered as we traverse them clockwise, for some integer $n$, where in any pair of adjacent areas, one is inside $\partial C$ and one is outside.
Inside each area $A_i$, the number of non-boundary edges is either even or odd. If it is even, half of these edges are oriented towards $v$ and half from $v$ because they all follow the same pattern. If the number is odd, let $j$ be the smallest positive integer such that $A_{i+j}$ also contains odd number of non-boundary edges, where $i+j$ is understood mod $2n$. If $j$ is even then $A_i$ and $A_{i+j}$ follow the same pattern and the number of edges between them is even, so exactly half of the edges in $A_i\cup A_{i+j}$ is oriented towards $v$ and half from $v$. If $j$ is odd then $A_i$ and $A_{i+j}$ follow different patterns and the number of edges between them is odd, which however also means that exactly half of the edges in $A_i\cup A_{i+j}$ is oriented towards $v$ and half from $v$.
Combined with the fact that due to the balanced orientation of $\partial C$, half of the boundary edges go into $v$ and half out of $v$, this means that the orientation at $v$ is balanced.

\end{proof}

Let us note that the first half of the proof uses neither that the lattice $\Lambda$ is symmetric nor that all degrees are even, so we can also use it to deduce the following technical lemma that we will heavily use in Section~\ref{section:archimedean}.

\begin{lemma}\label{lemma:withinboundaries}
Let $\Lambda$ be a planar lattice and suppose there is a (finitary) $\Aut(\Lambda)$-factor of iid which is almost surely a hierarchy on $\Lambda$. Then for every $k\in\mathbb{N}$, there is a (finitary) $\Aut(\Lambda)$-factor of iid which is almost surely a $k$-spaced hierarchy with the additional property that the boundary $\partial C$ of any cluster $C$ is a union of edge-disjoint cycles and any path between any $u\in C$, $v\in \cup_{n=1}^\infty C^{n+}$ must cross $\partial C$.
Subsequently, for any distinct clusters $B,C$, we have that $d(\partial B,\partial C)\geq k-\frac{\ell}{2}$, where $\ell$ is an upper bound on the number of vertices forming any face of $\Lambda$.
\end{lemma}

\begin{proof}
Given $k\in\mathbb{N}$, let us fix a $\max\{k,\frac{\ell+1}{2}\}$-spaced hierarchy $\mathbf{H}_k$ (whose existence is ensured by Proposition~\ref{prop:hierarchy}), and observe that we can repeat verbatim the proof of that for every vertex $v\in B$, the number of edges from $\partial B$ that are incident to $v$ is even. It also remains true that for any cluster $C\in \mathbf{H}_k$, the vertices $v\in C$ which remained outside of $\partial C$ are exactly those such that all faces they are in contain a vertex from $C^+$.
Let us repeat the reassignment of vertices outside $\partial C$ to $C^+$ so that for any cluster $C$,
$C=\{v\in V(\Lambda) : v \text{ is inside } \partial C \text{ but outside } \partial B \text{ for every } B \text{ such that } B^+=C\}\cup\partial C$ as required. In particular, this reassignment leaves the boundaries $\partial C$, $C\in \mathbf{H}_k$ intact and ensures the second property.

Finally, if $B$ and $C$ are distinct non-adjacent clusters then $d(\partial B, \partial C)\geq k$ simply because the hierarchy is $k$-spaced. On the other hand, for any $C\in \mathbf{H}_k$ and any vertex $u\in\partial C^+$, the definition of boundary tells us that $d(u,C^{++})\leq\frac{\ell}{2}$. But $d(\partial C, C^{++})\geq k$ as $C$ and $C^{++}$ do not neighbour, so by the triangle inequality,
\[k\leq d(\partial C, C^{++})\leq d(\partial C, \partial C^+)+d(\partial C^+, C^{++})\leq d(\partial C, \partial C^+)+\frac{\ell}{2}.\]
Thus $d(\partial C, \partial C^+)\geq k-\frac{\ell}{2}$ as required.
\end{proof}

\section{Schreier decorations of Archimedean lattices and $\Lambda_\square^d, d\geq3$ as finitary factors of iid}
\label{section:archimedean}

An \emph{Archimedean lattice} is a vertex-transitive tiling of the plane by regular polygons. It is known that there are ten of them (eleven if we count separately the two mirror images of the lattice $(3^4,6)$; see e.g.~\cite[Chapter 5]{BollRior}), out of which four have even regularity. These are the infinite grid $\Lambda_\square$ (that is the Cayley graph of $\mathbb{Z}^2$ with the standard generators), the triangular lattice $T$ (the only 6-regular planar lattice), the Kagom\'e lattice (i.e.\ the line graph of the hexagonal lattice $H$), and the 4-regular lattice $(3,4,6,4)$ where the numbers denote the orders of the faces when we go round a vertex.

\subsection{The square lattice}\label{subsection:square}

\begin{proof}[Proof of Theorem~\ref{thm:archimedean_schreier} for $\Lambda=\Lambda_\square$]
We start with independent $[0,1]$-labels on $V(\Lambda)$.

\begin{enumerate}[label=\textbf{Step \arabic* }]

\item\textbf{Spaced hierarchy.}
The square lattice has got 4-fold symmetry, so we can use Theorem~\ref{thm:latticehierarchy}, Proposition~\ref{prop:hierarchy} and Lemma~\ref{lemma:withinboundaries} to create a $k$-spaced hierarchy $\mathbf{H}_{2,k}$ based on percolation clusters where $k$ is a sufficiently large integer. Moreover, each part $C\in \mathbf{H}_{2,k}$ is assigned a number from $\{1,2\}$ such that every child has got the other number than its parent.

\item\textbf{Definition of guards.}
The \emph{guards} of an edge $e=uv\in E(\Lambda_\square)$ are the four of its six incident edges that are perpendicular to it. That is, the four incident edges that appear with $e$ in a $C_4$ (see Figure~\ref{fig:guards}).

\begin{figure}[ht]
\centering
\begin{tikzpicture}[node distance=1cm,
main node/.style={circle,minimum width=1.9pt, fill, inner sep=0pt,outer sep = 0pt},
red node/.style={circle,minimum width=2.2pt, fill,red, inner sep=0pt,outer sep = 0pt},
blue node/.style={circle,minimum width=2.2pt, fill,blue, inner sep=0pt,outer sep = 0pt},
square node/.style={draw,regular polygon,regular polygon sides=4,fill,inner sep=0.9pt,outer sep=0.1pt},
triangle node/.style={draw,regular polygon,regular polygon sides=3,fill,inner sep=0.8pt,outer sep=0.1pt}]   
      \node[main node] (v) at (0,0) [label=below:{ $u$}] {};
  \node[main node] (w) at (2,0) [label=below left:{\includegraphics[width=3.5mm]{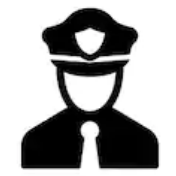}\scriptsize guard}] {};
  \node[main node] (u) at (-2,0)[label=below right:{\includegraphics[width=3.5mm]{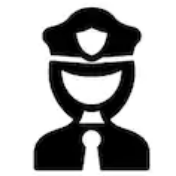}\scriptsize guard}] {};
  \node[main node] (b) at (0,2) [label=above:{ $v$}]{};
  \node[main node] (a) at (-2,2) [label=above right:{\includegraphics[width=3.5mm]{guardguy.png}\scriptsize guard}]{};
  \node[main node] (c) at (2,2) [label=above left:{\includegraphics[width=3.5mm]{guardgirl.png}\scriptsize guard}]{};
  \path [draw=gray,ultra thin] (u) edge (v);
  \path [draw=gray,ultra thin] (v) edge (w);
  \path [draw=gray,ultra thin] (a) edge (b);
  \path [draw=gray,thick] (v) edge (b);
  \path [draw=gray,ultra thin] (a) edge (u);
  \path [draw=gray,ultra thin] (b) edge (c);
  \path [draw=gray,ultra thin] (c) edge (w);
\end{tikzpicture}
\caption{Each edge has four guards.}
\label{fig:guards}
\end{figure}

\item\textbf{Red-blue edge colourings inside clusters.}
Each cluster numbered 1 wants to determine for itself an edge colouring consisting of monochromatic $C_4$-s.
That is, each such cluster $B$ finds a vertex $v\in B$ (e.g. the one with the largest label) and chooses one of the four $C_4$-s containing $v$ to have all its edges coloured red. As we want our $B$-pattern to be the union of blue $C_4$-s and red $C_4$-s and such that every $u\in B$ has two incident edges blue and two red, fixing one red $C_4$ determines the rest of the pattern. We call this colouring the \textit{inner pattern} of $B$.

\begin{figure}
\centering
\begin{subfigure}[b]{.4\linewidth}
    \centering
    \begin{tikzpicture}[node distance=1cm,
    main node/.style={circle,minimum width=1.9pt, fill, inner sep=0pt,outer sep = 0pt},
    red node/.style={circle,minimum width=2.2pt, fill,red, inner sep=0pt,outer sep = 0pt},
    blue node/.style={circle,minimum width=2.2pt, fill,blue, inner sep=0pt,outer sep = 0pt},
    square node/.style={draw,regular polygon,regular polygon sides=4,fill,inner sep=0.9pt,outer sep=0.1pt},
    triangle node/.style={draw,regular polygon,regular polygon sides=3,fill,inner sep=0.8pt,outer sep=0.1pt},
    scale=1.5]   
        \node[main node] (a0) at (0,0) {};
        \node[main node] (a1) at (0,1) {};
        \node[main node] (a2) at (0,2) {};
        \node[main node] (a3) at (0,3) {};
        
        \node[main node] (b0) at (1,0) {};
        \node[main node] (b1) at (1,1) {};
        \node[main node] (b2) at (1,2) {};
        \node[main node] (b3) at (1,3) {};
        
        \node[main node] (c0) at (2,0) {};
        \node[main node] (c1) at (2,1) {};
        \node[main node] (c2) at (2,2) {};
        \node[main node] (c3) at (2,3) {};
        
        \node[main node] (d0) at (3,0) {};
        \node[main node] (d1) at (3,1) {};
        \node[main node] (d2) at (3,2) {};
        \node[main node] (d3) at (3,3) {};
        
        \path [draw=red, dashed, thin] (a0) edge (a1);
        \path [draw=red, dashed, thin] (b0) edge (b1);
        \path [draw=red, dashed, thin] (c0) edge (c1);
        \path [draw=red, dashed, thin] (d0) edge (d1);
        
        \path [draw=blue, thin] (a1) edge (a2);
        \path [draw=blue, thin] (b1) edge (b2);
        \path [draw=blue, thin] (c1) edge (c2);
        \path [draw=blue, thin] (d1) edge (d2);
        
        \path [draw=red, dashed, thin] (a2) edge (a3);
        \path [draw=red, dashed, thin] (b2) edge (b3);
        \path [draw=red, dashed, thin] (c2) edge (c3);
        \path [draw=red, dashed, thin] (d2) edge (d3);

        \path [draw=red, dashed, thin] (a0) edge (b0);
        \path [draw=red, dashed, thin] (a1) edge (b1);
        \path [draw=red, dashed, thin] (a2) edge (b2);
        \path [draw=red, dashed, thin] (a3) edge (b3);
        
        \path [draw=blue, thin] (b0) edge (c0);
        \path [draw=blue, thin] (b1) edge (c1);
        \path [draw=blue, thin] (b2) edge (c2);
        \path [draw=blue, thin] (b3) edge (c3);
        
        \path [draw=red, dashed, thin] (c0) edge (d0);
        \path [draw=red, dashed, thin] (c1) edge (d1);
        \path [draw=red, dashed, thin] (c2) edge (d2);
        \path [draw=red, dashed, thin] (c3) edge (d3);

\end{tikzpicture}
\end{subfigure}\qquad%
\begin{subfigure}[b]{.4\linewidth}
    \centering
     \begin{tikzpicture}[node distance=1cm,
    main node/.style={circle,minimum width=1.9pt, fill, inner sep=0pt,outer sep = 0pt},
    red node/.style={circle,minimum width=2.2pt, fill,red, inner sep=0pt,outer sep = 0pt},
    blue node/.style={circle,minimum width=2.2pt, fill,blue, inner sep=0pt,outer sep = 0pt},
    square node/.style={draw,regular polygon,regular polygon sides=4,fill,inner sep=0.9pt,outer sep=0.1pt},
    triangle node/.style={draw,regular polygon,regular polygon sides=3,fill,inner sep=0.8pt,outer sep=0.1pt},
    scale=1.5]   
        \node[main node] (a0) at (0,0) {};
        \node[main node] (a1) at (0,1) {};
        \node[main node] (a2) at (0,2) {};
        \node[main node] (a3) at (0,3) {};
        
        \node[main node] (b0) at (1,0) {};
        \node[main node] (b1) at (1,1) {};
        \node[main node] (b2) at (1,2) {};
        \node[main node] (b3) at (1,3) {};
        
        \node[main node] (c0) at (2,0) {};
        \node[main node] (c1) at (2,1) {};
        \node[main node] (c2) at (2,2) {};
        \node[main node] (c3) at (2,3) {};
        
        \node[main node] (d0) at (3,0) {};
        \node[main node] (d1) at (3,1) {};
        \node[main node] (d2) at (3,2) {};
        \node[main node] (d3) at (3,3) {};
        
        \path [draw=blue, thin] (a0) edge (a1);
        \path [draw=blue, thin] (b0) edge (b1);
        \path [draw=blue, thin] (c0) edge (c1);
        \path [draw=blue, thin] (d0) edge (d1);
        
        \path [draw=red, dashed, thin] (a1) edge (a2);
        \path [draw=red, dashed, thin] (b1) edge (b2);
        \path [draw=red, dashed, thin] (c1) edge (c2);
        \path [draw=red, dashed, thin] (d1) edge (d2);
        
        \path [draw=blue, thin] (a2) edge (a3);
        \path [draw=blue, thin] (b2) edge (b3);
        \path [draw=blue, thin] (c2) edge (c3);
        \path [draw=blue, thin] (d2) edge (d3);

        \path [draw=red, dashed, thin] (a0) edge (b0);
        \path [draw=red, dashed, thin] (a1) edge (b1);
        \path [draw=red, dashed, thin] (a2) edge (b2);
        \path [draw=red, dashed, thin] (a3) edge (b3);
        
        \path [draw=blue, thin] (b0) edge (c0);
        \path [draw=blue, thin] (b1) edge (c1);
        \path [draw=blue, thin] (b2) edge (c2);
        \path [draw=blue, thin] (b3) edge (c3);
        
        \path [draw=red, dashed, thin] (c0) edge (d0);
        \path [draw=red, dashed, thin] (c1) edge (d1);
        \path [draw=red, dashed, thin] (c2) edge (d2);
        \path [draw=red, dashed, thin] (c3) edge (d3);

    \end{tikzpicture}
\end{subfigure}

    \caption{Examples of what the inner pattern can look like on a fixed 16-vertex square}
\end{figure}
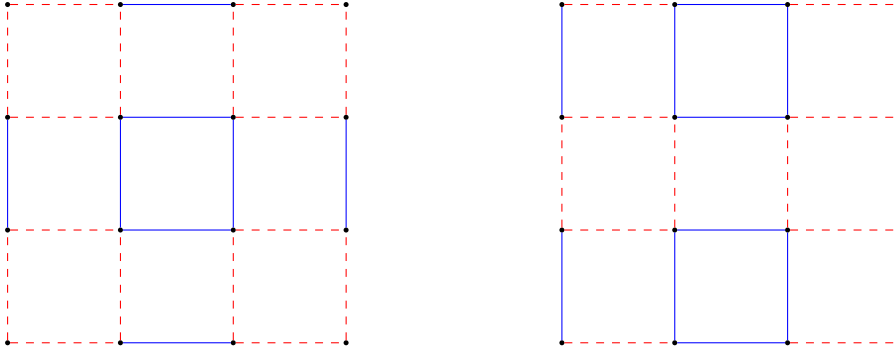

On the other hand, each cluster numbered 2 wants to determine for itself an edge colouring in which all parallel edges have got the same colour (and so every $C_4$ has got two edges red and two blue in an alternating manner, and every edge has different colour than its guards).
We call such a colouring an \emph{interface pattern}.

\item\textbf{Amalgamating the patterns.}
For every cluster $C$ numbered 2 (i.e.~a cluster with an interface pattern), we will apply its pattern to all edges with at least one endpoint in $C$. Let $B$ be a cluster numbered 1 and $e$ an edge with both endpoints in $B$.
Let $S$ be the set of the guards of $e$ whose other endpoint (the one not shared with $e$) is not in $B$ (note that $S$ might be empty).
Then we colour $e$ as follows: If there is $g\in S$ on which the inner pattern of $B$ (should it be extended there) disagrees with the interface pattern by which $g$ is coloured, we colour $e$ according to the interface pattern, i.e.~with the opposite colour than $g$ has. This may or may not coincide with the colour the inner pattern would give to $e$. If there is no such guard $g\in S$, $e$ is coloured according to the inner pattern of $B$.
Our hierarchy being $k$-spaced (where $k\geq4$) ensures that this is well defined.

\item \textbf{Claim that the colouring is balanced.} In the edge colouring defined above, every vertex $v \in V(\Lambda)$ has exactly two red and two blue incident edges a.s.
\begin{proof}[Proof of claim]
If $v\in C$ where $C$ is numbered 2, then all its four incident edges follow the same interface pattern. As any interface pattern is internally balanced, so are the colours at $v$.

So let us now assume that $v$ is in a cluster $B$ numbered 1, and consider the 9-vertex square centred on $v$.
Regardless of what clusters the remaining 8 vertices belong to, the final colouring of the four edges incident to $v$ can be described by saying that we first colour all four of them with the  inner pattern (so that a perpendicular pair is red and the complementary one blue), and then recolour if necessary.

First observe that thanks to the hierarchy being $k$-spaced where $k\geq5$, the 9-vertex square does not contain both a vertex from $B^+$ and a vertex from a $B^-$. Hence there is only one interface pattern according to which edges may wish to recolour themselves.
Also note that regardless of the inner and the interface pattern, exactly one red and exactly one blue edge out of the four edges incident to $v$ already follow the interface pattern at the beginning, so they will not get recoloured in any case. That is, at most two of the four edges will get recoloured.

If two edges get recoloured, then as explained above, one of them is initially blue and one is initially red. So recolouring both of them does not change the multiplicity of colours incident to $v$, and as this multiplicity was balanced to start with, it will remain balanced. If no edges get recoloured, then again, we invoke that the initial colouring was already balanced. Finally, we would like to argue that it cannot happen that exactly one of the four edges gets recoloured.

Let $uv$ and $bv$ be the two perpendicular edges on which  the inner pattern and the interface pattern disagree, and suppose on the contrary that only $uv$ gets recoloured. Note that $uv$ is a guard of $bv$, so if $u\notin B$ then \emph{both} $uv$ and $bv$ would get recoloured. Also $bv$ is a guard of $uv$, so by the same logic we also must have $b\in B$. If $a$, the fourth corner of the square spanned by $b$, $u$ and $v$, was not in $B$, then again \emph{both} $bv$ and $uv$ would get recoloured. Finally, the inner pattern and the interface pattern agree on the two remaining guards of $uv$, so they will not be recoloured in any case, meaning $uv$ also does not get recoloured. But that is in contradiction with the assumption that $uv$ does get recoloured. Hence it cannot happen that exactly one of the four edges incident to $v$ gets recoloured.

\begin{figure}[ht]
     \centering
         \begin{tikzpicture}[node distance=1cm,
main node/.style={circle,minimum width=1.9pt, fill, inner sep=0pt,outer sep = 0pt},
red node/.style={circle,minimum width=2.2pt, fill,red, inner sep=0pt,outer sep = 0pt},
blue node/.style={circle,minimum width=2.2pt, fill,blue, inner sep=0pt,outer sep = 0pt},
square node/.style={draw,regular polygon,regular polygon sides=4,fill,inner sep=0.9pt,outer sep=0.1pt},
triangle node/.style={draw,regular polygon,regular polygon sides=3,fill,inner sep=0.8pt,outer sep=0.1pt}]   
  \node[main node] (a) at (-1.4,1.4) [label=above:{\footnotesize $a\in B$}] {};
  \node[main node] (b) at (0,1.4) [label=above:{\footnotesize $b\in B$}] {};
  \node[main node] (c) at (1.4,1.4)[label=right:{}] {};
  \node[main node] (u) at (-1.4,0) [label=left:{\footnotesize $u\in B$}] {};
  \node[main node] (v) at (0,0) [label=below right:{\footnotesize $v\in B$}] {};
  \node[main node] (w) at (1.4,0)[label=right:{}] {};
  \node[main node] (x) at (-1.4,-1.4) [label=above:{}] {};
  \node[main node] (y) at (0,-1.4) [label=below:{}] {};
  \node[main node] (z) at (1.4,-1.4)[label=right:{}] {};
  \path [draw=gray, thin] (a) edge (b);
  \path [draw=gray, thin] (b) edge (c);
  \path [draw=gray, thin] (a) edge (u);
  \path [draw=gray, thin] (b) edge (v);
  \path [draw=gray, thin] (c) edge (w);
  \path [draw=gray, thin] (u) edge (v);
  \path [draw=gray, thin] (v) edge (w);
  \path [draw=gray, thin] (u) edge (x);
  \path [draw=gray, thin] (v) edge (y);
  \path [draw=gray, thin] (w) edge (z);
  \path [draw=gray, thin] (x) edge (y);
  \path [draw=gray, thin] (y) edge (z);
\end{tikzpicture}
        \caption{Vertex $v$ has two red and two blue incident edges.}
\end{figure}
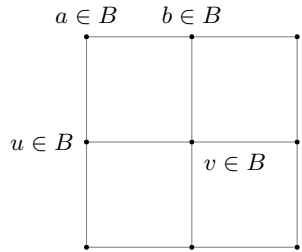

\end{proof}

\item \textbf{Claim that there are no infinite paths.}
There are no infinite monochromatic paths in the balanced colouring defined above a.s. That is, the set of blue edges is a union of vertex-disjoint finite blue cycles, and similarly the set of red edges is a union of vertex-disjoint finite red cycles.
\begin{proof}[Proof of claim]
Suppose there is an infinite monochromatic path $P$.
Let $A$ be a cluster whose intersection with $P$ is non-empty.
Then the path must share a vertex with every cluster in the sequence $(A^{n+})_{n=0}^\infty$, where $A^{0+}\vcentcolon=A$ and $A^{n+}\vcentcolon=\big(A^{(n-1)+}\big)^+$ for every positive integer $n$.
Pick the $n\in\{1,2\}$ such that $A^{n+}$ is numbered 1, and so gets the inner pattern.
There must be a section $u,v_1,\dots,v_m$ of $P$ such that $v_m\in\partial A^{n+}$, $v_i\in A^{n+}\setminus\partial A^{n+}$ for every $i\in[m-1]$, and $u\in\partial B$ for some cluster $B$ such that $B^+=A^{n+}$.
We used Lemma~\ref{lemma:withinboundaries} when building the $k$-spaced hierarchy,
so we can now deduce that $m\geq k-2$.
Moreover, for every $i\in[m-(k-5),m]$ and every cluster $C$ such that $C^+=A^{n+}$,
\[
k-2\leq d(C,\partial A^{n+})\leq d(C,v_i)+d(v_i,\partial A^{n+})\leq d(C,v_i)+k-5,
\]
and so $d(C,v_i)\geq3$.
This means that the edges $v_{m-(k-4)}v_{m-(k-5)},\dots,v_{m-1}v_{m}$
must have empty $S$, and so must all follow the inner pattern of $A^{n+}$. 
In particular, as $k\geq8$, this is at least four edges. But no monochromatic walk of length four in an inner pattern can contain five distinct vertices (i.e.~be a path) because inner patterns are made up of monochromatic $C_4$-s. This would be a contradiction, and so there is no infinite monochromatic path almost surely.
\end{proof}
\item\textbf{Orientation.}
Finally, let each monochromatic cycle choose one of the two strong orientations for itself (e.g.\ by finding the edge with the largest sum of labels and orienting it from the vertex with the larger label to the one with the smaller label). Then every vertex lies in one strongly oriented blue cycle and one strongly oriented red cycle, and so the coloured orientation is balanced.
\end{enumerate}
\end{proof}
\begin{remark}
In fact, it is not necessary to rely on Proposition~\ref{prop:hierarchy} to prove Theorem~\ref{thm:archimedean_schreier} for $\Lambda=\Lambda_\square$. It makes the proof neater, but one can also make do just with the basic percolation clusters for the matching pair $\{\Lambda_\square,\Lambda_\boxtimes\}$. The approach not using Proposition~\ref{prop:hierarchy} (in which every yellow cluster is merged with its green parent into a blob and edges bridging different blobs get an interface pattern) heavily relies on the fact that critical percolation on either of the lattices in this matching pair has no infinite cluster a.s.\ (though expected sizes of the clusters will be infinite) \cite{Russo}.
Interestingly, the related conjecture of there being no percolation at criticality (site or bond) in $\mathbb{Z}^d$ has not been settled for $3\leq d\leq18$ (see \cite{DNS} for a partial progress).
\end{remark}

\subsection{Grids in higher dimensions}

Our construction on $\Lambda_\square$ made use of the hierarchical structure of clusters, as well as of their thickness preventing local recolourings of edges close to boundaries and other tweaks to overlap. We achieved these desired properties by classical percolation-theoretic results, which are sadly not available in higher dimensions. 

However, one could also construct such clusters (arbitrarily spaced) using Borel-combinatorial results of Gao, Jackson, Krohne and Seward \cite{gao2015forcing, gao2018continuous}, as well as of Marks and Unger \cite{marks2017borel}.

In \cite{gao2015forcing} and \cite{gao2018continuous}, such hierarchical decompositions named (layered or unlayered) \textit{toasts} are defined. Theorem 5.5 in \cite{marks2017borel} (attributed to Gao, Jackson, Krohne, and Seward) states that any Borel action of $\mathbb{Z}^d$ has a Borel toast of arbitrary thickness. Even though we do not have a $\mathbb{Z}^d$ action on $[0,1]^{V(\Lambda_\square^d)}$, the proof in \cite[Appendix A]{marks2017borel} goes through with minor modifications. 
In fact, one only has to prove \cite[Lemma A.2]{marks2017borel} for $[0,1]^{V(\Lambda_\square^d)}$ up to nullsets, and the rest can be repeated verbatim.

\begin{lemma}[Factor of iid version of Lemma~A.2 from \cite{marks2017borel}]
For any sequence $r_1,r_2,\dots$ of integers, there exists a sequence $C_1,C_2,\dots$ of random  subsets of $V(\Lambda_\square^d)$ such that $C_i$ is a factor of iid maximal $r_i$-discrete subset of $V(\Lambda_\square^d)$, and for any $\varepsilon>0$ and $v\in V(\Lambda_\square^d)$,
\[
\mathbb{P}\left[~\left|\{i\in \mathbb{N} ~|~ B(v,\varepsilon r_i)\cap C_i \neq \emptyset \} \right|=\infty\right]=1.
\]
\end{lemma}

\begin{proof}
Since $\Lambda_\square^d$ is transitive, it is sufficient to prove the statement for a fixed vertex.

Let $r \in \mathbb{N}$ be fixed, and assume that each $v \in V(\Lambda_\square^d)$ has countably many iid labels $\big(l_1(v), l_2(v), \ldots\big)$. Let $D_0=\emptyset$, and define $D_j \subseteq V(\Lambda_\square^d)$ for $j \geq 1$ as follows:

\[v \in D_j \Leftrightarrow \Big( v\in D_{j-1} \Big) \textrm{ or } \Big(B_{\Lambda_\square^d}(v, r) \cap D_{j-1} = \emptyset \textrm{ and for all } u\in B_{\Lambda_\square^d}(v, r) \textrm{ we have } l_j(v)>l_j(u)\Big).\]

Each set $D_j$ is $r$-discrete, and therefore their increasing union $C = \cup_{j=1}^\infty D_j$ is also $r$-discrete. Moreover, $C$ is maximal $r$-discrete with probability 1, because if some $u \notin D_j$ could be added to $D_j$, then it is in $D_{j+1}$ with probability $1/ |B_{\Lambda_\square^d}\left(u, r\right)|$, which is bounded away from 0. If $u \notin C$ could be added to $C$, then an event with probability bounded away from 0 would have to have failed infinitely many times.

Now we argue that $\mathbb{P}[B_{\Lambda_\square^d}(v,\varepsilon r) \cap C \neq \emptyset] > c$, where $c>0$ is a constant independent of $r$. Indeed, with probability $\frac{|B_{\Lambda_\square^d}\left(v, \varepsilon r\right)|}{|B_{\Lambda_\square^d}\left(v, (1+\varepsilon) r\right)|}$, the maximum of $l_1$ within $B_{\Lambda_\square^d}\big(v, (1+\varepsilon) r\big)$ is obtained somewhere within $B_{\Lambda_\square^d}(v, \varepsilon r)$, and therefore $B_{\Lambda_\square^d}(v,\varepsilon r) \cap D_1 \neq \emptyset$, which also implies $B_{\Lambda_\square^d}\left(v,\varepsilon r\right) \cap C \neq \emptyset$. Thanks to the polynomial growth of $\Lambda_\square^d$, the ratio $\frac{|B_{\Lambda_\square^d}(v, \varepsilon r)|}{|B_{\Lambda_\square^d}\left(v, (1+\varepsilon) r\right)|}$ is bounded away from 0 by a constant independent of $r$.

Generating $C_i$ independently of each other using the above construction with $r=r_i$, we get the desired sequence of subsets. The set of indices $i$ where $B(v,\varepsilon r_i)\cap C_i \neq \emptyset$ is infinite because each $i$ is in it independently with probability at least $c$.
\end{proof}

The $C_i$ constructed above are in fact finitary factors. The rest of the proof of \cite[Theorem~5.5]{marks2017borel} implies the following. 

\begin{corollary} \label{cor:spaced_hierarchy_Z_d}
$\Lambda_\square^d$ admits a finitary factor of iid $k$-spaced hierarchy for any $k \in \mathbb{N}$.
\end{corollary}

While we were writing this paper up, Jan Grebík and Václav Rozhoň informed us that using techniques developed in \cite{conley2020borel}, they can prove the existence of (Borel) toasts in any Borel graph with connected components isomorphic to $\Lambda_\square^d$. This should allow one to construct Borel Schreier decorations on all such Borel graphs \cite{grebik2021toasts}.

\begin{proof}[Proof of Theorem~\ref{thm:Zd}]
We will first construct a balanced edge colouring with colours $c_1,\dots,c_d$, and then give orientation separately to each monochromatic cycle.
\begin{enumerate}[label=\textbf{Step \arabic* }]

\item\textbf{Spaced hierarchy.}
Let $\mathbf{H}$ be the finitary factor of iid hierarchy on $\Lambda_\square^d$ given by Corollary~\ref{cor:spaced_hierarchy_Z_d}.
and let us use Proposition~\ref{prop:hierarchy} to obtain from it $(\mathbf{H}_{2d-2,k},\eta:\mathbf{H}_{2d-2,k}\to[2d-2])$ where $k$ is a sufficiently large integer and $\mathbf{H}_{2d-2,k}$ is a $k$-spaced hierarchy.

\item\textbf{Edge colouring of even-numbered clusters.}
The edges of $\Lambda_\square^d$ travel in $d$ different directions. Let us observe that if we remove all edges in $d-2$ of them, we are left with disjoint copies of the square lattice.
Each cluster $C\in \mathbf{H}_{2d-2,k}$ such that $\eta(C)$ is odd
will be assigned an \emph{interface pattern} in direct analogy to subsection~\ref{subsection:square}. By an \emph{interface pattern}, we understand a balanced edge colouring in which the edge directions of $\Lambda_\square^d$ and the $d$ colours are in bijection. That is, after removing all edges of any $d-2$ directions (= all edges of any $d-2$ colours), we would end up with an interface pattern on disjoint copies of the square lattice.

In particular, each cluster $C\in \mathbf{H}_{2d-2,k}$ numbered $2d-2$ identifies its
ancestor $C^{(2d-2)+}$ (which is also numbered $2d-2$) and then all its
$2d-2$-th generation offspring $C=C_1,\dots,C_n$. The clusters $C_1,\dots,C_n$ will make a common choice of the interface pattern.

Subsequently, each $\{2,\dots,2d-4\}$-numbered cluster $C\in \mathbf{H}_{2d-2,k}$ identifies its nearest $2d-2$-numbered ancestor $C^{\rm above}: =C^{(2d-2-\eta(C))+}$ and should there be any, also its nearest $2d-2$-numbered offspring and their interface patterns. The role of the $\{2,\dots,2d-4\}$-numbered clusters, as well as of the odd-numbered clusters, is to provide a balanced transition between the patterns of the $2d-2$-numbered clusters.

If a $\{2,\dots,2d-4\}$-numbered cluster $C$ has no $2d-2$-numbered offspring then it will simply copy the pattern from $C^{\rm above}$.
Otherwise, we will assign interface patters so that for any cluster $C$ with $\eta(C)=2i\leq2d-4$, the directions of the colours $c_1,\dots,c_i$ in $C$ and $C^{\rm above}$ agree.
In particular, we will first assign patterns to 2-numbered clusters, then 4-numbered etc. all the way up to $2d-4$. A $\{2,\dots,2d-4\}$-numbered cluster $C$ shall identify the interface patterns of its grandchildren and of $C^{\rm above}$. If these agree in the direction of the colour $c_i$ then $C$ will simply adopt the pattern of its grandchildren without a change. If they disagree then $C$ will employ the interface pattern obtained from that of its grandchildren by swapping $c_i$ and the colour that currently travels in the direction that $c_i$ takes in $C^{\rm above}$.

\item\textbf{Edge colouring of odd-numbered clusters.}
Each cluster $C\in \mathbf{H}_{2d-2,k}$ such that $\eta(C)$ is odd
will be assigned an \emph{inner pattern}.
By an \emph{inner pattern}, we understand a balanced edge $d$-colouring in which all the edges in $d-2$ chosen directions have the same colour, while the remaining copies of the square lattice get a pattern consisting of monochromatic $C_4$-s of the remaining two colours (that is an inner pattern in the sense of subsection~\ref{subsection:square}).

In particular, each odd-numbered cluster $C\in \mathbf{H}_{2d-2,k}$ identifies its parent $C^+$ and its interface pattern and its children (should there be any) and their interface pattern.
By construction, these two interface patterns
are either the same or one can be obtained from the other by a single transposition, that is by swapping two colours.
In the former case and also when $C$ has no children, $C$ will randomly choose two colours and put monochromatic $C_4$-s to the grids spanned by their two directions. In the latter case, the $C_4$-s shall be built from the two colours on which the two interface patterns disagree. Anyhow, the colours and directions of the colours not chosen to form $C_4$-s will propagate through $C$ without change of direction.

\item\textbf{Boundaries between clusters.}
Whenever two clusters $C$ and $C^+$ meet, there are $d-2$ colours which simply travel in a fixed direction, which is moreover the same in $C$ and $C^+$. We will let these $d-2$ colours propagate in their directions also on all edges with one endpoint in $C$ and one in $C^+$. After disregarding these, we are left with copies of the square grid, and on every such copy, we see an inner pattern in the sense of subsection ~\ref{subsection:square} coming from one of the clusters and interface pattern from the other one. We amalgamate these into a balanced colouring as in subsection~\ref{subsection:square}.

Finally, we observe that for every $C\in \mathbf{H}_{2d-2,k}$ and $i\in[d]$, there is, with probability 1, an $n\in\mathbb{N}$ such that $C^{n+}$ has an inner pattern and one of the colours of its $C_4$-s is $c_i$. Therefore, there cannot be any infinite monochromatic path, and every colour class is almost surely a union of finite cycles.
To finish the construction, let each monochromatic cycle randomly choose one of the two strong orientations for itself.

\end{enumerate}

\end{proof}

\subsection{The triangular lattice}

Similarly as in the case of the higher-dimensional grids, once we have a spaced enough hierarchy, we can use the patterns developed for $\Lambda_\square$ to obtain a Schreier decoration of $T$.

\begin{proof}[Proof of Theorem~\ref{thm:archimedean_schreier} for $\Lambda=T$]
We will first construct a red-blue-green edge colouring and then as in the previous case, give orientation separately to each monochromatic cycle.
\begin{enumerate}[label=\textbf{Step \arabic* }]

\item\textbf{Spaced hierarchy.}
$T$ has got 3-fold symmetry and is self-matching, so Theorem~\ref{thm:matchlattice} tells us that its site-percolation critical probability is $\frac{1}{2}$ and percolation does not occur at criticality. Let $\mathbf{H}$ be the finitary factor of iid hierarchy on $T$ given by the percolation clusters, and let us use Proposition~\ref{prop:hierarchy} to obtain from it $(\mathbf{H}_{4,k},\eta:\mathbf{H}_{4,k}\to[4])$ where $k$ is a sufficiently large integer and $\mathbf{H}_{4,k}$ is a $k$-spaced hierarchy.

\item\textbf{Edge colouring of odd-numbered clusters.}
The edges of $T$ travel in three different directions. Let us observe that if we remove all edges in one of them, we are left with the square lattice.
Each cluster $C\in \mathbf{H}_{4,k}$ numbered 1 or 3 will be assigned an \emph{inner pattern}. By an \emph{inner pattern}, we understand a balanced edge 3-colouring in which all the edges in a chosen direction have the same colour, while the remaining square lattice gets a pattern consisting of monochromatic $C_4$-s of the remaining two colours (that is an inner pattern in the sense of subsection~\ref{subsection:square}). In particular, each cluster $C\in \mathbf{H}_{4,k}$ numbered 1 identifies its
grandparent $C^{++}$ (which is numbered 3) and then all its
grandchildren $C=C_1,\dots,C_n$. The clusters $C_1,\dots,C_n$ will make a common choice of the inner pattern. They randomly choose one of the three edge directions, colour all edges in the chosen direction green, and put blue and red $C_4$-s on the remaining square lattice.

Subsequently, each cluster $C\in \mathbf{H}_{4,k}$ numbered 3 identifies its grandchildren (should there be any) and their inner pattern and its grandparent $C^{++}$ and its inner pattern. Then it chooses the direction not chosen by either of the patterns and
colours all its edges blue.
If there is a choice between two directions, siblings (i.e.~clusters with the same parent $C^+$) make it together.

Every odd-numbered cluster $C$ will apply its chosen inner pattern to all edges which are inside $\partial C$, but outside the boundaries of its children.

\item\textbf{Edge colouring of even-numbered clusters.}
Each cluster $C\in \mathbf{H}_{4,k}$ numbered 2 or 4 will be assigned an \emph{interface pattern} in direct analogy to subsection~\ref{subsection:square}. By an \emph{interface pattern}, we understand a balanced edge colouring in which the edge directions of $T$ and the three colours are in bijection. That is, after removing all edges of any of the directions (= all edges of any of the colours), we would end up with an interface pattern on the square lattice.

In particular, each even-numbered cluster $C\in \mathbf{H}_{4,k}$ identifies its parent $C^+$ and its inner pattern and its children (should there be any) and their inner pattern. By construction, these two inner patterns disagree both in their chosen direction and their chosen colour. To the direction of $C^+$, $C$ will assign the colour chosen by $C^+$ (i.e.~green if $C^+$ is numbered 1 and blue if it is numbered 3). To the direction of the pattern of its children, $C$ will assign the colour chosen by the children. This fully determines the interface pattern of $C$ (if $C$ has no children, then it shall randomly choose one of the two suitable patterns).

\begin{figure}
\centering
\begin{subfigure}[p]{0.45\textwidth}
    \centering
    \begin{tikzpicture}[scale=1.3]
        \coordinate (0;0) at (0,0); 
        
        \foreach \c in {1,...,2}{%
        \foreach \i in {0,...,5}{%
        \pgfmathtruncatemacro\j{\c*\i}
        \coordinate (\c;\j) at (60*\i:\c);  
        } }
        
        \foreach \i in {0,2,...,10}{%
        \pgfmathtruncatemacro\j{mod(\i+2,12)}%
        \pgfmathtruncatemacro\k{\i+1}
        \coordinate (2;\k) at ($(2;\i)!.5!(2;\j)$);    }
        
        \path [draw=green, thin] (2;4) edge (2;2);
        \path [draw=green, thin] (2;5) edge (2;1);
        \path [draw=green, thin] (2;6) edge (2;0);
        \path [draw=green, thin] (2;7) edge (2;11);
        \path [draw=green, thin] (2;8) edge (2;10);
        
        \path [draw=green, red] (2;5) edge (2;4);
        \path [draw=green, red] (1;3) edge (1;2);
        \path [draw=green, red] (1;4) edge (0;0);
        \path [draw=green, red] (2;9) edge (1;5);
        
        \path [draw=green, red] (1;1) edge (2;2);
        \path [draw=green, red] (1;0) edge (2;1);
        \path [draw=green, red] (2;11) edge (2;0);
        
        \path [draw=green, red] (1;1) edge (2;3);
        \path [draw=green, red] (0;0) edge (1;2);
        \path [draw=green, red] (1;4) edge (1;3);
        \path [draw=green, red] (2;8) edge (2;7);
        
        \path [draw=green, red] (2;0) edge (2;1);
        \path [draw=green, red] (2;11) edge (1;0);
        \path [draw=green, red] (2;10) edge (1;5);
        
        \path [draw=green, blue] (2;6) edge (2;5);
        \path [draw=green, blue] (2;7) edge (1;3);
        \path [draw=green, blue] (2;8) edge (1;4);
        
        \path [draw=green, blue] (1;2) edge (2;3);
        \path [draw=green, blue] (0,0) edge (1;1);
        \path [draw=green, blue] (1;5) edge (1;0);
        \path [draw=green, blue] (2;10) edge (2;11);
        
        \path [draw=green, blue] (1;2) edge (2;4);
        \path [draw=green, blue] (1;3) edge (2;5);
        \path [draw=green, blue] (2;7) edge (2;6);
        
        \path [draw=green, blue] (2;1) edge (2;2);
        \path [draw=green, blue] (1;0) edge (1;1);
        \path [draw=green, blue] (1;5) edge (0;0);
        \path [draw=green, blue] (2;9) edge (1;4);
    \end{tikzpicture}     
    \caption{An example of an inner pattern }
    \label{fig:inner_pattern}
\end{subfigure}\qquad%
\begin{subfigure}[p]{0.45\textwidth}
    \centering
    \begin{tikzpicture}[scale=1.3]
        \pgfmathsetmacro{\c}{2} 
        \foreach \i in {0,...,5}{%
        \pgfmathtruncatemacro\j{\c*\i}
        \coordinate (\c;\j) at (60*\i:\c);
        } 
        \foreach \i in {0,2,...,10}{%
        
        \pgfmathtruncatemacro\j{mod(\i+2,12)}%
        \pgfmathtruncatemacro\k{\i+1}
        \coordinate (2;\k) at ($(2;\i)!.5!(2;\j)$) ;    }
        
        \path [draw=red, thin] (2;6) edge (2;4);
        \path [draw=red, thin] (2;7) edge (2;3);
        \path [draw=red, thin] (2;8) edge (2;2);
        \path [draw=red, thin] (2;9) edge (2;1);
        \path [draw=red, thin] (2;10) edge (2;0);
        
        \path [draw=green, thin] (2;4) edge (2;2);
        \path [draw=green, thin] (2;5) edge (2;1);
        \path [draw=green, thin] (2;6) edge (2;0);
        \path [draw=green, thin] (2;7) edge (2;11);
        \path [draw=green, thin] (2;8) edge (2;10);
        
        \path [draw=blue, thin] (2;8) edge (2;6);
        \path [draw=blue, thin] (2;9) edge (2;5);
        \path [draw=blue, thin] (2;10) edge (2;4);
        \path [draw=blue, thin] (2;11) edge (2;3);
        \path [draw=blue, thin] (2;0) edge (2;2);
    \end{tikzpicture} 
    \caption{An example of an interface pattern }
    \label{fig:interface}
\end{subfigure}
\label{fig:triangular_patterns}
\caption{Examples of patterns used in odd- and even-numbered clusters}
\end{figure}
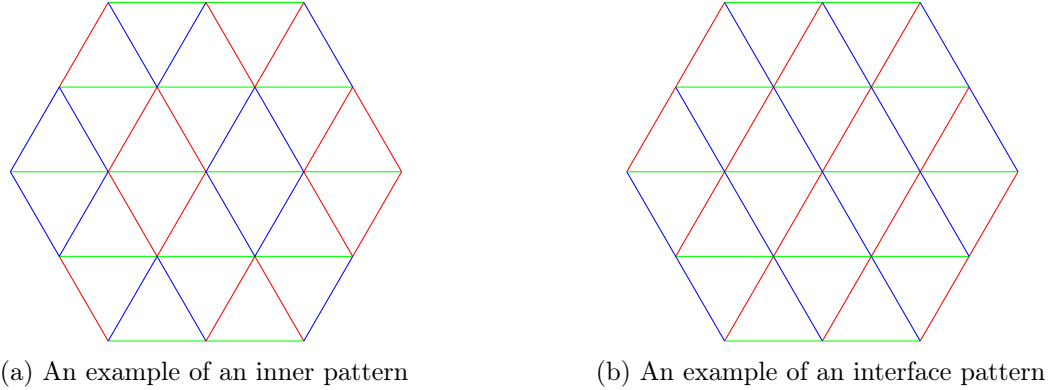

\item\textbf{Boundaries.}
Let $\partial C$, $C\in \mathbf{H}_{4,k}$ be the cluster boundaries as per the definition in subsection~\ref{subsec:percolationplanar}. For every $C\in \mathbf{H}_{4,k}$, we want the boundary $\partial C$ between $C$ and $C^+$ to travel only in the two directions not chosen by the odd-numbered cluster.
Suppose that $uv\in\partial C$ travels in the undesired direction and $x, y$ are the common neighbours of $u$ and $v$. Then by the definition of a boundary, exactly one of $x, y$ is in $C$ and one is in $C^+$ (without loss of generality, $x\in C^+$). We will replace every such $uv$ in the boundary with the pair $ux, xv$. As $\mathbf{H}_{4,k}$ is $k$-spaced and $k\geq3$, such changes cannot cause that boundaries of two different clusters would touch (share a vertex).

Now for a cluster $C$, there is exactly one colour $c$ such that there is a direction in which all the edges of both $C$ and $C^+$ have colour $c$. Moreover, the boundary $\partial C$ never travels in this direction, so after forgetting the edges of colour $c$, we are left with two neighbouring clusters in the square lattice, one of which has got the inner pattern and the other the interface pattern. We proved in subsection~\ref{subsection:square} that there is a 2-colouring of the boundary which amalgamates the two patterns into a balanced 2-colouring.

Finally, as in subsection~\ref{subsection:square}, there are no infinite monochromatic paths a.s. To finish the construction, let each monochromatic cycle randomly choose one of the two strong orientations for itself.

\end{enumerate}

\end{proof}

\subsection{The Kagomé and the $(3,4,6,4)$ lattice}

The remaining two Archimedean lattices differ from the first two ones in that not all faces are bounded by cycles of the same length (the lattice $(3,4,6,4)$ is not even edge-transitive). This slightly increased heterogeneity makes the proofs much easier, and in the latter case we even obtain an upper bound on the length of monochromatic cycles, meaning the resulting decoration is a Schreier graph of actions of many more groups than just $F_2$.

\begin{proof}[Proof of Theorem~\ref{thm:archimedean_schreier} for $\Lambda=K$]
Split the lattice into finite clusters organised in a hierarchy $\mathbf{H}$ as in Lemma~\ref{lemma:withinboundaries}. Each finite cluster picks a pattern composed of monochromatic (red and blue) triangles. There are two ways to randomly place this pattern on the Kagom\'e lattice. Consider the boundaries $\partial C$ where $C\in \mathbf{H}$. If a boundary travels on two out of the three edges of a given triangle $abc$, replace the two edges in the boundary (say $ab,bc\in\partial C$) by the third one ($ac$) -- as a side effect, each boundary becomes a union of \emph{vertex}-disjoint cycles.

Let each cluster $C$ apply its pattern to all non-boundary edges with at least one endpoint in $C$. If two neighbouring clusters $C$ and $C^+$ agree in their pattern, simply propagate it to $\partial C$ as well. If $C$ and $C^+$ disagree, we will colour the boundary as follows.
Suppose $uv$ is in $\partial C$ and $w$ is the common neighbour of $u$ and $v$. Then by construction, either all three edges $uv, vw, wu$ are in the boundary (in which case we will propagate the pattern of $C^+$ to them) or neither $vw$ nor $wu$ is in the boundary. In the latter case, $vw$ and $wu$ follow the same pattern and are in the same triangle, so they are assigned the same colour. Let $uv$ get the opposite colour.

As remarked earlier, every vertex $x\in C$ is incident to either two or none edges from $\partial C$. If it is incident to none and also if the patterns of $C$ and $C^+$ agree,
the colouring at $x$ is an inner pattern, which is balanced. If $x$ is incident to two boundary edges $xy$ and $xz$ and $yz\in E(K)$, the same reasoning applies. Finally, if the patterns disagree and $y$ and $z$ are not adjacent, let $N(y)$ be the common neighbour of $y$ and $x$ and $N(z)$ the common neighbour of $z$ and $x$. Then $y$ has the opposite colour than $N(y)$ by the definition of the colouring on boundaries, and similarly $z$ has the opposite colour than $N(z)$. Hence the amalgamated colouring is balanced at every vertex.

As in the previous cases, the colour classes are thus unions of finite cycles, and we let each monochromatic cycle randomly pick one of the two possible strong orientations to finish the decoration. 
\end{proof}

\begin{figure}[h!]
\centering
    \begin{subfigure}[b]{.4\linewidth}
        \centering
        \begin{tikzpicture}[node distance=1cm,
        main node/.style={circle,minimum width=1.9pt, fill, inner sep=0pt,outer sep = 0pt},
        red node/.style={circle,minimum width=2.2pt, fill,red, inner sep=0pt,outer sep = 0pt},
        blue node/.style={circle,minimum width=2.2pt, fill,blue, inner sep=0pt,outer sep = 0pt},
        square node/.style={draw,regular polygon,regular polygon sides=4,fill,inner sep=0.9pt,outer sep=0.1pt},
        triangle node/.style={draw,regular polygon,regular polygon sides=3,fill,inner sep=0.8pt,outer sep=0.1pt},
        scale=1.5]   
        \node[main node] (00) at (0,0 ) [label=right:{ $x$}] {};
        \node[main node] (f1) at (60:1) [label=right:{ $y$}] {} ;
        \node[main node] (f2) at (120:1) {};
        \node[main node] (a1) at (240:1) {};
        \node[main node] (a2) at (300:1) [label=right:{ $z$}]{};
        
        \path [draw=gray, thin] (00) edge (f1);
        \path [draw=red, dashed, thin] (00) edge (f2);
        \path [draw=red, dashed, thin] (f1) edge (f2);
        
        \path [draw=blue, thin] (00) edge (a1);
        \path [draw=gray, thin] (00) edge (a2);
        \path [draw=blue, thin] (a1) edge (a2);
        
        \end{tikzpicture}
        \caption{Both triangles incident to $x$ are in the same cluster.}
    \end{subfigure}\qquad%
    \begin{subfigure}[b]{.4\linewidth}
    \centering
        \begin{tikzpicture}[node distance=1cm,
        main node/.style={circle,minimum width=1.9pt, fill, inner sep=0pt,outer sep = 0pt},
        red node/.style={circle,minimum width=2.2pt, fill,red, inner sep=0pt,outer sep = 0pt},
        blue node/.style={circle,minimum width=2.2pt, fill,blue, inner sep=0pt,outer sep = 0pt},
        square node/.style={draw,regular polygon,regular polygon sides=4,fill,inner sep=0.9pt,outer sep=0.1pt},
        triangle node/.style={draw,regular polygon,regular polygon sides=3,fill,inner sep=0.8pt,outer sep=0.1pt},
        scale=1.5]   
        \node[main node] (00) at (0,0 ) [label=right:{ $x$}] {};
        \node[main node] (f1) at (60:1) [label=right:{ $y$}] {};
        \node[main node] (f2) at (120:1) {};
        \node[main node] (a1) at (240:1) [label=left:{ $z$}] {};
        \node[main node] (a2) at (300:1) {};
        
        \path [draw=gray, thin] (00) edge (f1);
        \path [draw=red, dashed, thin] (00) edge (f2);
        \path [draw=red, dashed, thin] (f1) edge (f2);
        
        \path [draw=gray, thin] (00) edge (a1);
        \path [draw=red, dashed, thin] (00) edge (a2);
        \path [draw=red, dashed,  thin] (a1) edge (a2);
        
        \end{tikzpicture}  
        \caption{One triangle incident to $x$ is in $C$ and one in $C^+$.}
        
    \end{subfigure}
\caption{Kagom\'e lattice -- two cases that can occur at a boundary vertex}
\label{fig:Kagome_boundary}
\end{figure}
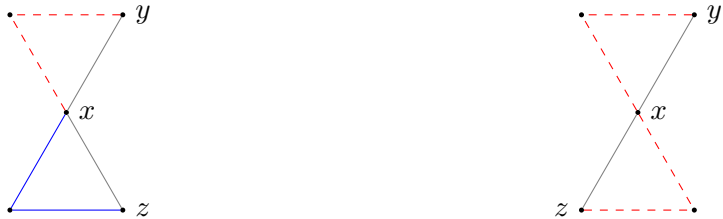

\begin{proof}[Proof of Theorem~\ref{thm:archimedean_schreier} for $\Lambda=(3,4,6,4)$]
The $(3,4,6,4)$ lattice is not edge-transitive which makes constructing a Schreier decoration easy in this case. Every edge is either part of a triangle or a hexagon, but not both.

Therefore, we can simply colour all triangles red and all hexagons blue. To complete the Schreier decoration, we choose a random strong orientation for each triangle and hexagon independently.
\end{proof}

Constructing a factor of iid perfect matching on the $(3,4,6,4)$ lattice is similarly easy: each hexagon can choose one of its two perfect matchings independently at random, and these matchings together form a perfect matching of the whole lattice.

\section{Ramifications for proper edge colourings and perfect matchings}\label{section:perfmatch_and_properedgecolr}

Proper edge $2d$-colourings encode actions of the group $(\mathbb{Z}/2\mathbb{Z})^{*2d}$, the $2d$-fold free product of $(\mathbb{Z}/2\mathbb{Z})$ with itself, just like Schreier decorations encode actions of the free group $F_d$. Both groups have $T_{2d}$ as their standard Cayley graph, and $T_{2d}$ is the universal cover of any $2d$-regular graph. In this sense, looking for these structures on $2d$-regular graphs is equally natural, but we focus more on Schreier decorations because all $2d$-regular graphs have a (deterministic) Schreier decoration, while this is not the case for proper edge $2d$-colourings.
When searching for these structures as factors of iid, however, we do not know of a single transitive example, where one exists and the other does not.

\subsection{Proper edge $2d$-colourings and perfect matchings}\label{subsec:properedgeperfmatch}
The absence of infinite monochromatic paths in our Schreier decorations allows us to immediately obtain proper edge $2d$-colourings whenever the graph in question is bipartite.
\begin{corollary}[of Theorems~\ref{thm:archimedean_schreier} and~\ref{thm:Zd}]
For every $d\geq2$, there is a finitary $Aut(\Lambda_\square^d)$-factor of iid which is a proper edge $2d$-colouring of $\Lambda_\square^d$ almost surely. Subsequently, there is a finitary $Aut(\Lambda_\square^d)$-factor of iid which is a perfect matching of $\Lambda_\square^d$ almost surely \cite{Adam}.
\end{corollary}
\begin{proof}
The factors from Theorems~\ref{thm:archimedean_schreier} and~\ref{thm:Zd} partition the edges of $\Lambda_\square^d, d\geq2$ into finite monochromatic cycles. $\Lambda_\square^d$ is bipartite, so all of these cycles are even. To obtain a proper $2d$-colouring, let each cycle of colour $c_i, i\in[d]$ choose randomly one of the two proper 2-colourings of its edges by $c_i^{\rm light}$ and $c_i^{\rm dark}$.

By choosing one colour class (say $c_1^{\rm light}$), we obtain the second statement.
\end{proof}

In fact, one can now obtain many more corollaries simply by noticing when graphs have got a decomposition into several locally identifiable copies of $\Lambda_\square^d$. The square lattice with diagonals added is an example of a transitive graph that decomposes into three copies of the square lattice.

\begin{corollary}
There is an $\Aut(\Lambda_\boxtimes)$-factor of  $\left([0,1]^{V(\Lambda_\boxtimes)},\mu_\boxtimes\right)$ which is a Schreier decoration a.s.\ Moreover, it has almost surely no infinite monochromatic paths. Subsequently, there is an $\Aut(\Lambda_\boxtimes)$-factor of $\left([0,1]^{V(\Lambda_\boxtimes)},\mu_\boxtimes\right)$ which is a perfect matching of $\Lambda_\boxtimes$ a.s.
\end{corollary}

If care is taken while constructing a Schreier decoration, bipartiteness is, however, not necessary not even when $G$ is not made up of $\Lambda_\square^d$-s.
\begin{corollary}[of Proposition~\ref{prop:HxP_whenH_hasPM}]
Let $H$ be a finite $2d-2$-regular graph whose chromatic index is $\chi'(H)=2d-2$. Then there is a finitary $Aut(H\times P)$-factor of iid which is a proper edge $2d$-colouring of $H\times P$ almost surely.
\end{corollary}
\begin{proof}
Let us in the proof of Proposition~\ref{prop:HxP_whenH_hasPM} always choose a matching $M$ which is given by a colour class of a proper edge colouring.
We made sure in the proof to spell out that all cycles of colour $c_1$ have even length. On each of those, pick randomly one of the two proper edge colourings with colours $c_1^{\rm light}$ and $c_1^{\rm dark}$.

Now on each connected component left after removing the decorated edges which is isomorphic to $H$, pick randomly a proper edge colouring with colours $c_3,\dots,c_{2d}$. On each component isomorphic to $(H\setminus M)\times\{v_i,v_{i+1}\}$,
let the $|V(H)|$ edges between $(H\setminus M)\times\{v_i\}$ and $(H\setminus M)\times\{v_{i+1}\}$ all have colour $c_3$, while we properly colour both $(H\setminus M)\times\{v_i\}$  and $(H\setminus M)\times\{v_{i+1}\}$ with $c_4,\dots,c_{2d}$.
\end{proof}

We can again obtain a perfect matching of $H\times P$ by choosing one of the colour classes, but there is already the trivial construction of simply picking a perfect matching on each of the $P$ copies of $H$.

\subsection{Line graphs}

\begin{proposition}
Let $G$ be a $2d$-regular graph. Then if $G$ admits a factor-of-iid Schreier decoration, so does its line graph $L(G)$.
\end{proposition}
\begin{proof}
If $d=1$ then $G=L(G)$, so the statement is a tautology.

If $d\geq2$ then there is a 1-to-1 correspondence between vertices of $G$ and the cliques $K_{2d}$ in $L(G)$. Also every vertex $v$ in $L(G)$ is in exactly two cliques $K_{2d}$ because it has two endpoints in $G$.

Suppose now that $K$ is isomorphic to $K_{2d}$ and its vertices are $c_1^{\rm in}, c_1^{\rm out},\dots, c_d^{\rm in}, c_d^{\rm out}$. Let us fix a proper edge $2d-1$-colouring of $K$ with colours $c'_1,\dots,c'_{2d-1}$. Then put on top an orientation such that for all $i\in[d], j\in[2d-1]$, the edge of colour $c'_j$ incident to $c_i^{\rm in}$ is oriented towards $c_i^{\rm in}$ if and only if the edge of colour $c'_j$ incident to $c_i^{\rm out}$ is oriented from $c_i^{\rm out}$.
This decorated $K$ will serve as a template for decorating the cliques of $L(G)$.

Indeed, the Schreier decoration of $G$ gives rise to a vertex $d$-colouring of $L(G)$ with colours $c_1,\dots,c_d$ such that moreover, every vertex of $L(G)$ labels one of the cliques it is in as `in' and the other as `out'. Also every clique $K_{2d}$ in $L(G)$ has vertices which are \emph{relatively to this clique} labelled $c_1^{\rm in}, c_1^{\rm out},\dots, c_d^{\rm in}, c_d^{\rm out}$. Given this vertex decoration, let every $K_{2d}$ in $L(G)$ decorate its edges as dictated by the template $K$. Since every $c_i^{\rm in}$ is $c_i^{\rm out}$ in its other clique, this gives a Schreier decoration of $L(G)$.
\end{proof}

Interestingly, we cannot use the same strategy to show that a proper edge $2d$-colouring of $G$ implies a proper edge $4d-2$-colouring of $L(G)$. This is exactly because every vertex $v$ of the line graph would get a decoration which is the same as viewed from either of the two cliques $v$ is in. However, when $d$ is even, we get the following.

\begin{proposition}
Let $G$ be a $2d$-regular graph where $d$ is an even positive integer. Then if $G$ admits a factor-of-iid balanced orientation, its line graph $L(G)$ has a factor-of-iid perfect matching.
\end{proposition}
\begin{proof}
Every vertex $v$ of $G$ gives rise to $K_{2d}$ in $L(G)$ with $d$ vertices labelled `in' and $d$ labelled `out'. As $d$ is even, in each such clique, we can pick a (random) perfect matching on the vertices labelled `in' -- crucially, these are labelled `out' with respect to the other cliques they are in. We claim that this is a perfect matching. Similarly as before, the balanced orientation of $G$ gives a vertex decoration of $L(G)$ such that every vertex is labelled exactly once with `in' and once with `out', therefore at every vertex, there is exactly one matched edge to be chosen.
\end{proof}

\section{Open questions and remarks}
\label{section:open_questions}

\begin{question}\label{qtn:equival}
Is it true that for all $2d$-regular graphs $G$ which are vertex- or edge-transitive, the following are equivalent?
\begin{enumerate}
    \item There is a factor of iid which is a proper edge $2d$-colouring of $G$ a.s.
    \item There is a factor of iid which is a perfect matching of $G$ a.s.
    \item There is a factor of iid which is a Schreier decoration of $G$ a.s.
    \item There is a factor of iid which is a balanced orientation of $G$ a.s.
\end{enumerate}
\end{question}

We have demonstrated in this paper that all four of the structures exist as factors of iid for $\Lambda_\square^d, d\geq2$, for $H\times P$ where $H$ is a finite $2d$-regular graph with $\chi'(H)=2d$, and more loosely speaking for graphs that are made up of locally identifiable copies of these. We have also constructed Schreier decorations and balanced orientations on more planar lattices than $\Lambda_\square$ and believe that the hierarchies which are available on them should make it possible to obtain proper edge $2d$-colourings too. Also pointing in this direction is the fact that all amenable Cayley graphs have got invariant random perfect matchings \cite{ConleyKechrisTucker, csoka2017invariant} -- the lattices $T$ and $K$ are indeed Cayley graphs of $G_T=\langle a,b,c ~|~ a^3=b^3=c^3=abc=e\rangle$ and $G_K=\langle a,b ~|~ a^3=b^3=(ab)^3=e\rangle$ respectively.

On the other hand, $H\times P$ where $|V(H)|$ is odd forms the prominent example of a class of graphs where none of the four factors exist. We have shown in Proposition~\ref{prop:finite_times_Z} that no balanced orientation of $H\times P$ is a factor of iid, and the proof that there is no factor-of-iid perfect matching follows similar lines. Indeed, suppose a perfect matching $M$ of $H\times P$ is given, and let $n(i)$ be the number of edges in $M$ of the form $(u,v_i)(u,v_{i+1})$. Then for every $i\in\mathbb{Z}$, if $n(i)$ is odd then $n(i+1)$ must be even and if $n(i)$ is even then $n(i+1)$ must be odd. This is, however, in contradiction with the existence correlation decay in factor of iid processes.

In our parallel paper \cite[Definition 22]{ourselves}, we further show a class $\mathcal{C}^*_{2d}=\{G^* : G \text{ is }2d\text{-regular}\}$ of $2d$-regular bipartite graphs for which a proper edge $2d$-colouring exists as a factor of iid if and only if perfect matching does if and only if Schreier decoration does. All the above leads us to tentatively conjecture that the answer to Question~\ref{qtn:equival} is yes for amenable graphs.

It also seems that if the transitivity assumption is relaxed, the four structures are listed in a strictly decreasing order of difficulty.
As we said, $H\times P$ where $|V(H)|$ is odd has not got a factor-of-iid balanced orientation, which by \cite[Lemma 25]{ourselves} means $(H\times P)^*\in\mathcal{C}^*_{2d}$ has no factor-of-iid perfect matching. Thus $(H\times P)^*$ has no factor-of-iid edge $2d$-colouring either by \cite[Proposition 5]{ourselves}. The non-transitivity of $(H\times P)^*$, however, allows for locally recognisable factorisation of the graph into copies of $K_{d,2d}$, which implies the existence of factor-of-iid balanced orientation whenever $d$ is even.

To get graphs with a factor-of-iid Schreier decoration but without a factor-of-iid perfect matching, we can start with any $2d$-regular $G$
satisfying Question~\ref{qtn:equival} in its positive sense and by attaching two pendant $K_{2d+5}^-$ to every vertex obtain a $2d+4$-regular graph which has no perfect matching at all.
Alternatively, we can construct graphs like the one in Figure~\ref{fig:fiid_Sch_no_fiid_pm} which have deterministic, but not factor-of-iid perfect matchings. Any perfect matching contains two or no edges of the big squares in an alternating fashion, which again violates correlation decay. However, Schreier decoration can be constructed by colouring the two edges to the left of a cut vertex red and to the right blue, or vice versa, chosen randomly and independently at each cut vertex. Each remaining finite component properly extends this 2-colouring and monochromatic cycles get a random strong orientation.

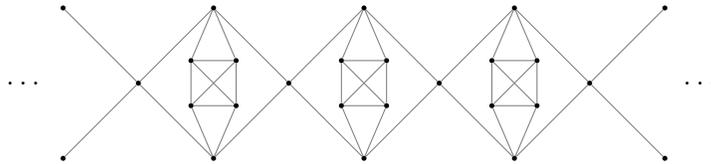
\begin{figure}[ht] 
     \centering
         \begin{tikzpicture}[node distance=1cm,
main node/.style={circle,minimum width=1.9pt, fill, inner sep=0pt,outer sep = 0pt},
red node/.style={circle,minimum width=2.2pt, fill,red, inner sep=0pt,outer sep = 0pt},
blue node/.style={circle,minimum width=2.2pt, fill,blue, inner sep=0pt,outer sep = 0pt},
square node/.style={draw,regular polygon,regular polygon sides=4,fill,inner sep=0.9pt,outer sep=0.1pt},
triangle node/.style={draw,regular polygon,regular polygon sides=3,fill,inner sep=0.8pt,outer sep=0.1pt}]   

  \node at (-4.5,0) {\ldots};
  
  \node[main node] (w) at (-4,1)[label=right:{}] {};  
  \node[main node] (z) at (-4,-1)[label=right:{}] {};
  \node[main node] (a) at (-3,0)[label=right:{}] {};
  \node[main node] (b) at (-2,1)[label=right:{}] {};
  \node[main node] (b_1) at (-2.3,0.3)[label=right:{}] {};
  \node[main node] (b_2) at (-1.7, 0.3)[label=right:{}] {};
  \node[main node] (c_1) at (-2.3,-0.3)[label=right:{}] {};
  \node[main node] (c_2) at (-1.7,-0.3)[label=right:{}] {};
  \node[main node] (c) at (-2,-1)[label=right:{}] {};
  \node[main node] (d) at (-1,0)[label=right:{}] {};
  \node[main node] (e) at (0,1)[label=right:{}] {};
  \node[main node] (e_1) at (-0.3,0.3)[label=right:{}] {};
  \node[main node] (e_2) at (0.3, 0.3)[label=right:{}] {};
  \node[main node] (f_1) at (-0.3,-0.3)[label=right:{}] {};
  \node[main node] (f_2) at (0.3,-0.3)[label=right:{}] {};
  \node[main node] (f) at (0,-1)[label=right:{}] {};
  \node[main node] (g) at (1,0)[label=right:{}] {};
  \node[main node] (h) at (2,1)[label=right:{}] {};
  \node[main node] (h_1) at (1.7,0.3)[label=right:{}] {};
  \node[main node] (h_2) at (2.3, 0.3)[label=right:{}] {};
  \node[main node] (i_1) at (1.7,-0.3)[label=right:{}] {};
  \node[main node] (i_2) at (2.3,-0.3)[label=right:{}] {};
  \node[main node] (i) at (2,-1)[label=right:{}] {};
  \node[main node] (j) at (3,0)[label=right:{}] {};
  \node[main node] (k) at (4,1)[label=right:{}] {};
  \node[main node] (l) at (4,-1)[label=right:{}] {};
  
  \node at (4.5,0) {\ldots};
  
  \path [draw=gray, thin] (a) edge (w);
  \path [draw=gray, thin] (a) edge (z);
  
  \path [draw=gray, thin] (a) edge (b);
  \path [draw=gray, thin] (a) edge (c);
  \path [draw=gray, thin] (b) edge (b_1);
  \path [draw=gray, thin] (b) edge (b_2);
  \path [draw=gray, thin] (c) edge (c_1);
  \path [draw=gray, thin] (c) edge (c_2);
  \path [draw=gray, thin] (b_1) edge (b_2);
  \path [draw=gray, thin] (b_1) edge (c_1);
  \path [draw=gray, thin] (b_1) edge (c_2);
  \path [draw=gray, thin] (b_2) edge (c_1);
  \path [draw=gray, thin] (b_2) edge (c_2);
  \path [draw=gray, thin] (c_1) edge (c_2);
  \path [draw=gray, thin] (d) edge (b);
  \path [draw=gray, thin] (d) edge (c);
  
  \path [draw=gray, thin] (d) edge (e);
  \path [draw=gray, thin] (d) edge (f);
  \path [draw=gray, thin] (e) edge (e_1);
  \path [draw=gray, thin] (e) edge (e_2);
  \path [draw=gray, thin] (f) edge (f_1);
  \path [draw=gray, thin] (f) edge (f_2);
  \path [draw=gray, thin] (e_1) edge (e_2);
  \path [draw=gray, thin] (e_1) edge (f_1);
  \path [draw=gray, thin] (e_1) edge (f_2);
  \path [draw=gray, thin] (e_2) edge (f_1);
  \path [draw=gray, thin] (e_2) edge (f_2);
  \path [draw=gray, thin] (f_1) edge (f_2);

  \path [draw=gray, thin] (g) edge (e);
  \path [draw=gray, thin] (g) edge (f);
  
  \path [draw=gray, thin] (g) edge (h);
  \path [draw=gray, thin] (g) edge (i);
  \path [draw=gray, thin] (h) edge (h_1);
  \path [draw=gray, thin] (h) edge (h_2);
  \path [draw=gray, thin] (i) edge (i_1);
  \path [draw=gray, thin] (i) edge (i_2);
  \path [draw=gray, thin] (h_1) edge (h_2);
  \path [draw=gray, thin] (h_1) edge (i_1);
  \path [draw=gray, thin] (h_1) edge (i_2);
  \path [draw=gray, thin] (h_2) edge (i_1);
  \path [draw=gray, thin] (h_2) edge (i_2);
  \path [draw=gray, thin] (i_1) edge (i_2);
  \path [draw=gray, thin] (j) edge (h);
  \path [draw=gray, thin] (j) edge (i);
  
  \path [draw=gray, thin] (j) edge (k);
  \path [draw=gray, thin] (j) edge (l);
  \end{tikzpicture}
  \caption{An infinite graph which admits a factor-of-iid Schreier decoration but not a factor-of-iid perfect matching}
  \label{fig:fiid_Sch_no_fiid_pm}
\end{figure}

Finally, any finite $2d$-regular graph has a factor-of-iid Schreier decoration, so whenever its chromatic index equals $2d+1$, we again get an example in which a proper edge $2d$-colouring does not exist, but when a perfect matching does, like in the Meredith graph, we get an instance for the last type which has exactly three of the four structures as factors of iid.

\begin{question}
Is there a $2d$-regular quasi-transitive graph which has a factor-of-iid proper edge $2d$-colouring or perfect matching, but not a factor-of-iid Schreier decoration or balanced orientation?
\end{question}

We also point out that all the Schreier decorations constructed in this paper have no infinite monochromatic paths, i.e.~no infinite orbits.

\begin{question} \label{qtn:sch_infinite_monochrom}
Is there a factor of iid Schreier decoration on the square lattice that has infinite monochromatic paths with positive probability? Or on any of the other lattices we consider? Or on any transitive graph? 
\end{question}
By ergodicity, this would imply that the factor of iid Schreier decoration has infinite monochromatic paths with probability 1. The next question aims to improve our coun\-ter\-ex\-am\-ples.

\begin{question} \label{qtn:non_Z_non_example}
Give an example of a transitive graph that is not quasi-isometric to $\mathbb{Z}$ and has no factor of iid balanced orientation. Or has no factor of iid Schreier decoration.
\end{question}

In fact, the best would be to have a precise structural description of when these decorations exist as factors of iid.
\begin{question}\label{qtn:characterisation}
Give a necessary and sufficient condition for a $2d$-regular transitive graph to have a factor of iid balanced orientation and/or Schreier decoration.
\end{question}

\bibliography{hivatkozat}
\bibliographystyle{plain}

\end{document}